\documentclass[11pt]{article}
\usepackage{graphicx}
\usepackage{fullpage}
\usepackage{amsmath}
\usepackage{amsthm}
\usepackage{thmtools, thm-restate}
\usepackage{amssymb,bm,verbatim,bbm,mathrsfs}
\usepackage{tikz}
\usepackage{enumerate}
\usepackage[hidelinks]{hyperref}
\usepackage{color}
\usepackage[font=small,labelfont=bf]{caption}
\usepackage{cleveref}
\usepackage{esvect}
\usetikzlibrary{decorations.markings}

\declaretheorem[numberwithin=section]{theorem}
\declaretheorem[sibling=theorem]{lemma}
\declaretheorem[sibling=theorem]{corollary}
\declaretheorem[sibling=theorem]{proposition}
\declaretheorem[sibling=theorem]{remark}
\declaretheorem[sibling=theorem]{definition}

\newcommand{\bb}[1]{\mathbb{#1}}
\newcommand{\cc}[1]{\mathcal{#1}}

\newcommand{\ob}[1]{\left(#1\right )} 
\newcommand{\cb}[1]{\left[#1\right ]} 
\newcommand{\set}[1]{\left\{#1\right\}} 
\newcommand{\ab}[1]{\langle #1 \rangle} 
\newcommand{\abs}[1]{\left\vert#1\right\vert} 
\newcommand{\restr}[1]{\!\!\upharpoonright_{#1}} 

\newcommand{\R}{\bb R}
\newcommand{\Z}{\bb Z}

\newcommand{\C}{\bb C}

\renewcommand{\P}{\bb P}

\newcommand{\Walks}[0]{\Gamma} 
\newcommand{\SWalks}[0]{\Omega} 
\newcommand{\si}[1]{#1\cdot #1} 
\newcommand{\mi}[2]{#1\cdot #2} 
\newcommand{\turn}[0]{\angle} 

\newcommand{\CSWalks}[0]{\SWalks^{c}} 
\newcommand{\nbWalks}[0]{\Walks_{\mathrm{nb}}} 
\newcommand{\tfWalks}[0]{\Walks_\mathrm{tf}} 
\newcommand{\bCSWalks}[0]{\bar{\SWalks}^{c}} 
\newcommand{\tnbWalks}[0]{\Walks_{\mathrm{tnb}}} 
\newcommand{\surface}[0]{\cc S} 
\newcommand{\LE}[0]{\mathrm{LE}} 
\newcommand{\orient}[1]{\vv{#1}} 
\newcommand{\loopadd}[0]{\oplus} 

\title{Ising Model Observables and Non-Backtracking Walks}
\author{Tyler Helmuth\footnote{Department of Mathematics, The
    University of British Columbia, Vancouver, B.C., Canada, V6T
    1Z2. Email: jhelmt@math.ubc.ca. Research partially supported by an
    NSERC Alexander Graham Bell Canada Graduate Scholarship.}}

\begin{document}
\maketitle
\date{}

\begin{abstract}
  This paper presents an alternative proof of the connection between
  the partition function of the Ising model on a finite graph $G$ and
  the set of non-backtracking walks on $G$. The techniques used also
  give formulas for spin-spin correlation functions in terms of
  non-backtracking walks. The main tools used are Viennot's theory of
  heaps of pieces and turning numbers on surfaces.
\end{abstract}

\section{Introduction}
\label{sec:introduction-results}

The connection between the Ising model and non-backtracking walks
began with a paper of Kac and Ward~\cite{KacWard1952}. A clearer
understanding of this connection emerged in a paper of
Sherman~\cite{Sherman1960}. Roughly contemporaneous to Sherman's work
was that of Vdovichenko~\cite{Vdovichenko1965a}, while later articles
include~\cite{Burgoyne1963, daCostaMaciel2003, Glasser1970}. The first
definitive understanding of the Kac-Ward approach appears to be~\cite{DolbilinZinovevMishchenkoShtankoShtogrin1999}.  More recently
Loebl~\cite{Loebl2004} and Cimasoni~\cite{Cimasoni2010} have obtained
generalizations of the expressions of Sherman and Kac-Ward to
arbitrary finite graphs by embedding them in closed orientable
surfaces of sufficiently high genus.

This paper gives a new self-contained and essentially elementary proof
that the partition function of the Ising model on a finite graph can
be expressed in terms of weighted non-backtracking walks. As in
previous works an embedding of the finite graph
into a surface of sufficiently high genus is required. Correlations
in Ising models on a graph $G$ can be computed by adding edges to
$G$ and taking derivatives with respect to the Ising couplings of the
added edges. By utilizing the expressions for the partition function on
the (typically non-planar) augmented graph we obtain new
representations for correlation functions in terms of non-backtracking
walks. In the case of planar graphs, these correlation representations
have been established simultaneously and independently by Kager, Lis,
and Meester~\cite{KagerLisMeester2013}, who have also given
convergence results for the planar correlation expressions.

The core of this work is a geometric interpretation of the Mayer expansion
for a hard core system, very much in the spirit of the isomorphism
between the hard core gas and branched
polymers~\cite{BrydgesImbrie2003}. By the high temperature expansion
the partition function for the Ising model on a finite part of the
honeycomb lattice is the partition function $Z$ of a system of cycles
with the hard core constraint that no pair of cycles share an
edge. The standard theory of the Mayer expansion then provides a
formula for $\log Z$.  It is not yet widely appreciated that Viennot's
theory of heaps of pieces~\cite{Viennot1986} is an efficient source of
formulas for Mayer coefficients.  Upon specialization to the honeycomb
lattice the results of this paper immediately show that $\log Z$ is a
sum over non-backtracking walks. On a general finite graph it is less
obvious how to identify high temperature graphs with disjoint cycles
that share no edges, but this can be done.  A more serious issue is
the role of topology when $G$ is not planar. The problem is easily
expressed: two closed smooth curves in the plane intersect an even
number of times, but on the torus, for example, this is false.  To
solve this problem following the literature referenced above
requires a descent into difficult topology, starting with spin
structures.  Here we have tried to shorten this journey, basing our
results on~\cite{CairnsMcIntyre1993}.

\subsection{Statement of Results}
\label{sec:statement-results}

The \emph{Ising model} on a graph $G=(V(G),E(G))$ with
\emph{couplings} $\{L_{xy}\}_{xy\in E(G)}$, $L_{xy}\in \R$, is the
probability measure $\P$ on configurations $\sigma\in\{-1,1\}^{V(G)}$ given
by
\begin{equation*}
  \P(\sigma) = \frac{1}{Z} \exp\ob{\sum_{xy\in
    E(G)}L_{xy}\sigma_{x}\sigma_{y}}\qquad Z = \sum_{\sigma}\exp\ob{\sum_{xy\in
    E(G)}L_{xy}\sigma_{x}\sigma_{y}}.
\end{equation*}
The quantity $\sigma_{x}$ for $x\in V(G)$ is called the \emph{spin} at
the vertex $x$, and $Z$ is called the \emph{partition function}. A
subgraph $H$ of $G$ is called \emph{even} if the degree of each vertex
of $H$ is even. Let $\cc E(G)$ denote the set of even subgraphs of
$G$. It is well known that the study of the Ising model is equivalent
to the study of the generating function of even subgraphs.
\begin{proposition}[name = High Temperature Expansion, see,
  e.g.,~\cite{Baxter1982}]
  \label{prop:HTE}
  The partition function $Z$ of the Ising model on the graph $G$
  can be expressed as
  \begin{equation}
    Z = 2^{\abs{V(G)}}\prod_{xy\in E(G)}\cosh L_{xy}\sum_{H\in\cc E(G)}
    \prod_{xy\in E(H)}\tanh L_{xy}.
  \end{equation}
\end{proposition}
From this point forward the factor $2^{\abs{V(G)}}\prod_{xy\in E(G)}\cosh
L_{xy}$ will be dropped as it plays no role in the quantities of
interest in the study of the Ising model.

The central result of this paper is a new proof of the fact that the
generating function for even subgraphs of a finite graph $G$
can be expressed as a linear combination of exponentials of sums of weighted
non-backtracking walks in the graph $G$. Proposition~\ref{prop:HTE}
immediately implies the same holds for the Ising model. 

To state our results some terminology will be needed. A \emph{walk of
  length $n$} is a sequence $\gamma = (\gamma_{1}, \ldots,
\gamma_{n})$ of adjacent vertices in a graph, and $\abs{\gamma} \equiv
n-1$ is the number of \emph{steps} of $\gamma$. A walk is
\emph{non-backtracking} if $\gamma_{j}\neq \gamma_{j+2}$ for $1\leq
j\leq \abs{\gamma}-1$. Let $\nbWalks(G,x,y)$ denote the set of
non-backtracking walks from $x$ to $y$ in $G$. A walk $\gamma$ is
\emph{closed} if $\gamma_{1}=\gamma_{n}$, and a closed walk $\gamma$
is \emph{tail free} if the walk $(\gamma_{1}, \dots, \gamma_{n},
\gamma_{2})$ is non-backtracking. Let $\tfWalks(G,x)$ denote the set
of tail free walks with initial vertex $x$, and $\tfWalks(G) =
\cup_{x}\tfWalks(G,x)$. See~\Cref{fig:TF}.

\tikzstyle{ivertex}=[circle,draw=black!100,fill=white, inner
sep = 0pt, minimum size=2mm]
\tikzstyle{overtex}=[circle,draw=black!100,fill=white, inner
sep = 0pt, minimum size=2mm]
\tikzset{->-/.style={decoration={
  markings,
  mark=at position #1 with {\arrow{>}}},postaction={decorate}}}
\tikzstyle{svertex}=[circle,draw=black!100,fill=black!60,inner
sep=0pt,minimum size =1mm]
\tikzstyle{vertex}=[circle,draw=black!100,fill=black!60, inner
sep = 0pt, minimum size=2mm]
\tikzstyle{cvertex}=[rectangle,draw=black!100,fill=black!60, inner
sep = 0pt, minimum size=2mm]
\tikzset{lgraph/.style={color=black,thick,dotted}}
\tikzset{sgraph/.style={color=black,thick,dashed}}

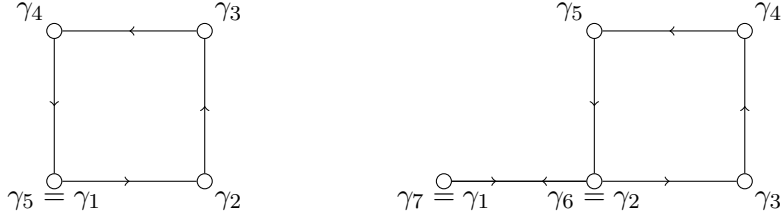
\begin{figure}
  \centering
  \beginpgfgraphicnamed{Figure1}
  \begin{tikzpicture}
    \node[ivertex] (v1) at (0,0) {};
    \node[ivertex] (v2) at (2,0) {};
    \node[ivertex] (v3) at (2,2) {};
    \node[ivertex] (v4) at (0,2) {};
    \node at (v1) [below] {$\gamma_{5}=\gamma_{1}$};
    \node at (v2) [below right] {$\gamma_{2}$};
    \node at (v3) [above right] {$\gamma_{3}$};
    \node at (v4) [above left] {$\gamma_{4}$};
    \draw[->-=.5] (v1) -- (v2);
    \draw[->-=.5] (v2) -- (v3);
    \draw[->-=.5] (v3) -- (v4);
    \draw[->-=.5] (v4) -- (v1);
  \end{tikzpicture} \qquad\qquad\qquad
  \endpgfgraphicnamed
  \beginpgfgraphicnamed{Figure2}
  \begin{tikzpicture}
    \node[ivertex] (v0) at (-2,0) {};
    \node[ivertex] (v1) at (0,0) {};
    \node[ivertex] (v2) at (2,0) {};
    \node[ivertex] (v3) at (2,2) {};
    \node[ivertex] (v4) at (0,2) {};
    \node at (v0) [below] {$\gamma_{7}=\gamma_{1}$};
    \node at (v1) [below ] {$\gamma_{6}=\gamma_{2}$};
    \node at (v2) [below right] {$\gamma_{3}$};
    \node at (v3) [above right] {$\gamma_{4}$};
    \node at (v4) [above left] {$\gamma_{5}$};
    \draw[->-=.33] (v0) -- (v1);
    \draw[->-=.33] (v1) -- (v0);
    \draw[->-=.5] (v1) -- (v2);
    \draw[->-=.5] (v2) -- (v3);
    \draw[->-=.5] (v3) -- (v4);
    \draw[->-=.5] (v4) -- (v1);
  \end{tikzpicture}
  \endpgfgraphicnamed
 \caption{The left-hand illustration depicts a tail free walk. The
   right-hand illustration depicts a walk that is not tail free.}
  \label{fig:TF}
\end{figure}

Let $\surface$ denote a surface of genus $g$. Suppose the graph $G$
properly embeds in $\surface$, meaning the vertices of $G$ can be
identified with points in $\surface$, and the edges of $G$ are curves
in $\surface$ that join the vertices and are mutually disjoint, except
possibly at their endpoints. For $\alpha\in
H_{1}(\surface,\Z_{2})$, define a weight on tail free walks by
\begin{equation}
  \label{eq:Introduction.1}
  w_{\alpha}(\gamma) \equiv (-1)^{\tau(\gamma)+\ab{\alpha,\gamma}}
  \prod_{j=1}^{\abs{\gamma}} K_{\gamma_{j}\gamma_{j+1}}, \qquad
  K_{xy}\equiv \tanh L_{xy}.
\end{equation}
In~\eqref{eq:Introduction.1} the term $\ab{\alpha,\gamma}$ is the
intersection pairing, i.e., the signed number of times $\alpha$ and
$\gamma$ intersect transversely, and $\tau(\gamma)$ is the turning
number of $\gamma$. In the case $\surface=\R^{2}$, $\tau(\gamma)$ is
the integral number of revolutions of the tangent vector to the walk;
the general definition of $\tau(\gamma)$ is given in
Section~\ref{sec:Turning-Surfaces}. The central result of this paper
is

\begin{restatable}{theorem}{ThmOne}
  \label{thm:Introduction-Main-Even}
  Let $G$ be a graph properly embedded in a surface $\surface$ of
  genus $g$. The generating function of even subgraphs of $G$ is given by
  \begin{equation*}
    \sum_{H\in\cc E(G)}\prod_{xy\in E(H)}K_{xy} = \frac{1}{2^{g}}
    \sum_{\alpha\in H_{1}(\surface,\Z_{2})} (-1)^{c(\alpha)}\exp\ob{ -\sum_{\gamma\in
      \tfWalks(G)} \frac{w_{\alpha}(\gamma)}{2\abs{\gamma}}}.
  \end{equation*}
\end{restatable}

The constants $c(\alpha)\in\{0,1\}$ are given explicitly in
Equation~\eqref{eq:Constant-c}. This result is not new:
see~\cite{Cimasoni2010}, Theorem~2.1 and Corollary~2.2. The main
advantage of the formula here is that it requires no sophisticated
concepts from topology.

Let $G^{\star}$ be the graph dual to $G$, meaning that $G^{\star}$ has
one vertex for each connected component of $\surface\setminus G$, and
vertices corresponding to neighbouring connected components are joined
by an edge. The connected components of $\surface\setminus G$ are
called \emph{faces}. Let $a,b$ be vertices in $G$, and let
$a^{\star},b^{\star}$ be vertices in $G^{\star}$ such that the faces
corresponding to $a^{\star}$ and $b^{\star}$ are incident to $a$ and
$b$. Let $\eta$ be a path from $a^{\star}$ to $b^{\star}$ in
$G^{\star}$, and define
\begin{equation*}
    \tilde K_{xy} =
    \begin{cases}
      \phantom{-}K_{xy} & (xy)^{\star}\notin \eta \\
      -K_{xy} & (xy)^{\star}\in\eta
    \end{cases}.
\end{equation*}
Referring to~\eqref{eq:Introduction.1}, let $\tilde
w_{\alpha}(\gamma)$ denote the weight of a walk when
$K_{\gamma_{j}\gamma_{j+1}}$ is replaced by $\tilde
K_{\gamma_{j}\gamma_{j+1}}$. For $\gamma\in\nbWalks(G,a,b)$ extend
$\gamma$ to begin at $a^{\star}$ and end at $b^{\star}$ and define the turning
number of $\gamma$ to be the turning number of
$\gamma\eta$, the concatenation of $\gamma$ with $\eta$.

\begin{restatable}{theorem}{ThmTwo}
  \label{thm:Introduction-Correlation-Formula-HT}
  Let $G$ be a graph properly embedded in a surface $\surface$ of
  genus $g$. For $a,b$ not nearest neighbours the spin-spin
  correlation function $\ab{\sigma_{a}\sigma_{b}}$ is given by
    \begin{equation}
      \label{eq:Introduction-Correlation-Formula-HT}
      \ab{\sigma_{a}\sigma_{b}} = -\frac{1}{2^{g}Z(G)}
      \sum_{\alpha\in H_{1}(\surface,\Z_{2})} (-1)^{c(\alpha)}
    \sum_{\gamma\in\nbWalks(G,a,b)} \tilde w_{\alpha}(\gamma) \exp\ob{
    -\sum_{\gamma\in\tfWalks(G)}\frac{\tilde w_{\alpha}(\gamma)}
    {2\abs{\gamma}}},
    \end{equation}
    where the turning number of closed tail free walks is computed by
    viewing the graph $G$ as embedded in $\surface$.
\end{restatable}

The formula simplifies greatly for planar graphs:

\begin{restatable}{corollary}{ThmThree}
  \label{cor:High-Temperature-Planar-Correlation}
  Let $G$ be a graph embedded in $\R^{2}$. The spin-spin correlation
  $\ab{\sigma_{a}\sigma_{b}}$ is given by
  \begin{equation*}
    \ab{\sigma_{a}\sigma_{b}} = -\sum_{\gamma\in\nbWalks(G,a,b)}
    \tilde w(\gamma) \exp\ob{-\sum_{\gamma\in
        \tfWalks(G)}\frac{\tilde w(\gamma)-w(\gamma)}{2\abs{\gamma}}}
  \end{equation*}
\end{restatable}

Before giving the proofs of our theorems, let us make contact with the
spinor holomorphic fermion observable used
in~\cite{ChelkakHonglerIzyurov2012}.
Equation~\eqref{eq:Introduction-Correlation-Formula-HT} combined with
another correlation identity presented in
Section~\ref{sec:Applications} yields
\begin{restatable}{corollary}{ThmFour}
  \label{cor:Spinor-Identification}
  Let $F(a,b)$ denote the spinor holomorphic fermion observable
  of~\cite{ChelkakHonglerIzyurov2012} evaluated at $b$. Then, up to a
  modulus one multiplicative constant, $F(a,b) =
  \sum_{\gamma\in\nbWalks(G,a,b)} \tilde w(\gamma)$
\end{restatable}

The modulus one constant in Corollary~\ref{cor:Spinor-Identification}
can be explicitly identified given a choice of the corner the spinor
holomorphic fermion observable is evaluated at, and a choice of the
path $\eta$ used in the definition of the weight $\tilde
w$.

\section{Proofs}
\label{sec:proofs}

\subsection{Outline of Proofs}
\label{sec:outline-proofs}

Broadly speaking, the proofs utilize two main tools. The first is a
generalization of the turning number of a regular curve in the plane
that applies to regular curves in surfaces. The generalization to
surfaces was originally introduced in~\cite{Reinhart1960,
  Reinhart1963, Chillingworth1972}, though the self-contained
presentation given in Section~\ref{sec:Loop-Weights} is based on the
work of Cairns and McIntyre~\cite{CairnsMcIntyre1993}. The turning
number enables the self-intersections of a curve to be counted mod
2. The second tool is the theory of heaps of pieces introduced by
Viennot~\cite{Viennot1986}, which is briefly recalled
in Appendix~\ref{sec:Heaps}. 

After dealing with preliminary definitions and conventions in
Section~\ref{sec:Preliminaries} the paper proceeds in four main
steps:
\begin{itemize}
\item Section~\ref{sec:Even-Decomposition} introduces
  \emph{loops}, which are a type of orientable subgraph, and
  establishes the equality of the generating function for even
  subgraphs with a weighted generating function of collections of
  edge-disjoint loops. The results presented here are not new, but are
  included to keep the paper self contained.

\item To study the weighted generating function of collections of
  loops, Section~\ref{sec:Loop-Weights} considers loops as curves in a
  surface $\surface$. Utilizing the turning number of a curve on a
  surface, the weight on collections of loops is factored
  into a sum of weights, each of which is multiplicative on
  collections of loops.

\item The multiplicative weights obtained in the preceding step are
  well-suited to combinatorial analysis, and in
  Section~\ref{sec:TNB-Walks} the theory of heaps of pieces is used in
  combination with a loop-erasure argument to show that the generating
  function of multiplicatively weighted loops is equivalent to a sum
  of weighted generating functions of tail free walks.

\item In Section~\ref{sec:Applications} the preceding results are used
  to establish the results about the Ising model.
\end{itemize}

\subsection{Preliminaries}
\label{sec:Preliminaries}

A \emph{graph} will mean a finite graph without loops or multiple
edges, and $V=V(G)$, $E=E(G)$ will denote the sets of vertices and edges
of a graph $G$. The cardinality of a set $X$ will be denoted
$\abs{X}$. Edges $\{x,y\}$ of a graph will be abbreviated $xy$. A
graph is called \emph{even} if each vertex has even degree.

Given an oriented surface $\surface$ and a graph $G$, an
\emph{embedding} of $G$ is an identification of the vertices of $G$
with distinct points of $\surface$, and an identification of the edges
of $G$ with curves on $\surface$. It will be assumed that the curves
representing edges have no coincident segments. An embedding is
\emph{proper} if the curves representing edges intersect only at
vertices. If $\surface$ carries a Riemannian metric, then at each
vertex $v\in V(G)$ of an embedded graph $G$ there is a cyclic order on
the edges containing $v$ given by the orientation of the
surface. Throughout this paper $G$ will always denote a graph that is
properly embedded in an oriented surface $\surface$ that carries a
Riemannian metric.

\subsection{Decomposing Even Graphs}
\label{sec:Even-Decomposition}

This section defines a one-to-many map from even subgraphs to
collections of loops. This is done by choosing a matching of the edges
incident to a vertex $v$ for each vertex $v$. A choice of matchings
decomposes an even subgraph into a collection of closed loops, each
loop being formed by picking a sequence of edges that are matched to
one another.

In Section~\ref{sec:Even-Matching} a weight on matchings is defined
so that the sum of the weights of matchings corresponding to a given
even subgraph is equal to the weight of the subgraph itself. The
weight on matchings is transferred to a weight on collections of loops
in Section~\ref{sec:Loops-Matchings}. The important conclusion is
Lemma~\ref{lem:Decomposition-Weight}.

\subsubsection{Matchings of Even Graphs}
\label{sec:Even-Matching}

\begin{definition}
  The \emph{line graph} $\cc L(G)$ of a graph $G$ is the graph with
  vertices $E(G)$ and edges $\{ \{xy,yz\} \mid xy,yz\in E(G)\}$. 
\end{definition}

The vertices of the line graph will be called \emph{half-edges};
one can think of a vertex in $\cc L(G)$ as being the midpoint
of an edge in $G$. For a vertex $v$ in $G$ let $\iota(v,G)$ denote the
subgraph induced in $\cc L(G)$ by the set of vertices $\{ vw\in
E(G)\}$. See Figure~\ref{fig:Line-Graph}. If $G$ is an embedded graph,
there are embeddings of $\iota(v,G)$ given by identifying the
half-edges of $\iota(v,G)$ with points on the edges incident to $v$ in
a neighbourhood of $v$.

\begin{figure}[h]
  \centering
  \beginpgfgraphicnamed{Figure3}
  \begin{tikzpicture}[scale=3] 
    \draw[step=.5cm,gray,thin] (-.8,-.8) grid (.8,.8); 
    \node[cvertex] (v01) at (0.0,0.25) {};
    \node[vertex] (v02) at (0.0,0.75) {};
    \node[cvertex] (v10) at (0.25,0.0) {};
    \node[vertex] (v20) at (0.75,0.0) {};
    \node[cvertex] (v-10) at (-0.25,0.0) {};
    \node[vertex] (v-20) at (-0.75,0.0) {};
    \node[cvertex] (v0-1) at (0.0,-0.25) {};
    \node[vertex] (v0-2) at (0.0,-0.75) {};
    \node[vertex] (v11) at (0.5,0.25) {};
    \node[vertex] (v-11) at (-.5,0.25) {};
    \node[vertex] (v-1-1) at (-.5,-.25) {};
    \node[vertex] (v1-1) at (.5,-.25) {};
    \node[vertex] (v12) at (.25,.5) {};
    \node[vertex] (v-12) at (-.25,.5) {};
    \node[vertex] (v-1-2) at (-.25,-.5) {};
    \node[vertex] (v1-2) at (.25,-.5) {};
    \node[]  at (0,0) [below] {$\,\,\,\,\,0$};
    \draw[thick,white] [-] (v10) -- (v20);
    \draw[thick,white] [-] (v10) -- (v11);
    \draw[thick,white] [-] (v10) -- (v1-1);
    \draw[thick,white] [-] (v10) -- (v01);
    \draw[thick,white] [-] (v10) -- (v-10);
    \draw[thick,white] [-] (v10) -- (v0-1);
    \draw[thick,white] [-] (v-10) -- (v01);
    \draw[thick,white] [-] (v-10) -- (v0-1);
    \draw[thick,white] [-] (v-10) -- (v-20);
    \draw[thick,white] [-] (v-10) -- (v-11);
    \draw[thick,white] [-] (v-10) -- (v-1-1);
    \draw[thick,white] [-] (v-12) -- (v-11);
    \draw[thick,white] [-] (v-12) -- (v01);
    \draw[thick,white] [-] (v-12) -- (v02);
    \draw[thick,white] [-] (v-12) -- (v12);
    \draw[thick,white] [-] (v12) -- (v02);
    \draw[thick,white] [-] (v12) -- (v01);
    \draw[thick,white] [-] (v12) -- (v11);
    \draw[thick,white] [-] (v-1-2) -- (v-1-1);
    \draw[thick,white] [-] (v-1-2) -- (v0-1);
    \draw[thick,white] [-] (v-1-2) -- (v0-2);
    \draw[thick,white] [-] (v-1-2) -- (v1-2);
    \draw[thick,white] [-] (v1-2) -- (v0-2);
    \draw[thick,white] [-] (v1-2) -- (v0-1);
    \draw[thick,white] [-] (v1-2) -- (v1-1);
    \draw[thick,white] [-] (v-20) -- (v-11);
    \draw[thick,white] [-] (v-20) -- (v-1-1);
    \draw[thick,white] [-] (v-11) -- (v-1-1);
    \draw[thick,white] [-] (v20) -- (v11);
    \draw[thick,white] [-] (v20) -- (v1-1);
    \draw[thick,white] [-] (v11) -- (v1-1);
    \draw[thick,white] [-] (v0-2) -- (v0-1);
    \draw[thick,white] [-] (v0-1) -- (v01);
    \draw[thick,white] [-] (v01) -- (v02);
    \draw[lgraph] [-] (v10) -- (v20);
    \draw[lgraph] [-] (v10) -- (v11);
    \draw[lgraph] [-] (v10) -- (v1-1);
    \draw[sgraph] [-] (v10) -- (v01);
    \draw[sgraph] [-] (v10) -- (v-10);
    \draw[sgraph] [-] (v10) -- (v0-1);
    \draw[sgraph] [-] (v-10) -- (v01);
    \draw[sgraph] [-] (v-10) -- (v0-1);
    \draw[lgraph] [-] (v-10) -- (v-20);
    \draw[lgraph] [-] (v-10) -- (v-11);
    \draw[lgraph] [-] (v-10) -- (v-1-1);
    \draw[lgraph] [-] (v-12) -- (v-11);
    \draw[lgraph] [-] (v-12) -- (v01);
    \draw[lgraph] [-] (v-12) -- (v02);
    \draw[lgraph] [-] (v-12) -- (v12);
    \draw[lgraph] [-] (v12) -- (v02);
    \draw[lgraph] [-] (v12) -- (v01);
    \draw[lgraph] [-] (v12) -- (v11);
    \draw[lgraph] [-] (v-1-2) -- (v-1-1);
    \draw[lgraph] [-] (v-1-2) -- (v0-1);
    \draw[lgraph] [-] (v-1-2) -- (v0-2);
    \draw[lgraph] [-] (v-1-2) -- (v1-2);
    \draw[lgraph] [-] (v1-2) -- (v0-2);
    \draw[lgraph] [-] (v1-2) -- (v0-1);
    \draw[lgraph] [-] (v1-2) -- (v1-1);
    \draw[lgraph] [-] (v-20) -- (v-11);
    \draw[lgraph] [-] (v-20) -- (v-1-1);
    \draw[lgraph] [-] (v-11) -- (v-1-1);
    \draw[lgraph] [-] (v20) -- (v11);
    \draw[lgraph] [-] (v20) -- (v1-1);
    \draw[lgraph] [-] (v11) -- (v1-1);
    \draw[lgraph] [-] (v0-2) -- (v0-1);
    \draw[sgraph] [-] (v0-1) -- (v01);
    \draw[lgraph] [-] (v01) -- (v02);
  \end{tikzpicture}
  \endpgfgraphicnamed
  \caption{Part of the line graph $\cc L(G)$ superimposed on
    $G=\Z^{2}$.  The subgraph $\iota (0,G)$, which is not properly
    embedded in $\R^{2}$, is comprised of the square vertices and
    dashed edges.}
  \label{fig:Line-Graph}
\end{figure}
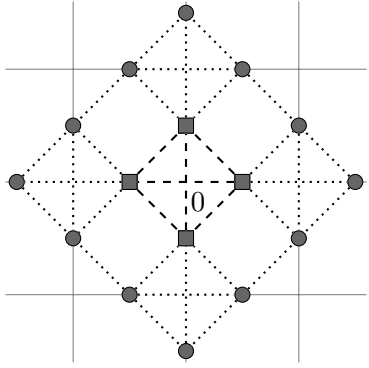

Figure~\ref{fig:Line-Graph} demonstrates that the embedding of
$\iota(v,G)$ obtained by placing the half-edges on the edges of $G$
need not be properly embedded, even if $G$ is properly embedded. The
next definitions and lemma make this precise. The coincidence of the
notions of crossing and intersection will be important in the sequel,
though this section will only use the notion of a crossing.

\begin{definition}
  Two edges $e_{1}e_{2},f_{1}f_{2}\in\iota(v,G)$ \emph{cross at $v$}
  if $e_{1}<f_{1}<e_{2}<f_{2}$ in the cyclic order on edges at $v$.
\end{definition}

\begin{definition}
  Two curves $\gamma_{1},\gamma_{2}$ \emph{intersect} at $v$ if there
  is a homeomorphism of a neighbourhood of $v$ such that $v$ is mapped
  to $0$, $\gamma_{1}$ is mapped to the $x$-axis, and $\gamma_{2}$ is
  mapped to the $y$ axis.
\end{definition}

\begin{lemma}
  \label{lem:Crossing-Intersecting}
   If $G$ is properly embedded then two edges $e_{1}e_{2}$ and $f_{1}f_{2}$ in
  $\iota(v,G)$ cross if and only if the curve $e_{1}e_{2}$ intersects
  the curve $f_{1}f_{2}$.
\end{lemma}

\begin{proof}
  Consider a circle $C(r)$ of radius $r$ centered at $v$, and let
  $x_{i}$ (resp.\ $y_{i}$) be the points where $e_{i}$ (resp.\ $f_{i}$)
  intersects $C(r)$. The edges $e_{1}e_{2}$ and $f_{1}f_{2}$ cross if
  and only if these points alternate around the perimeter of $C(r)$;
  as the edge $e_{1}e_{2}$ partitions $C(r)$ into two sets for $r$
  sufficiently small, it follows that an intersection occurs if and
  only if a crossing occurs.
\end{proof}

\begin{definition}
  A \emph{(perfect) matching} of $G=(V,E)$ is a
  collection $M\subset E$ of edges such that for any vertex $v\in V$ there
  exists a unique edge $e\in M$ containing $v$. Let $\cc M(G)$ denote
  the set of perfect matchings of $G$.
\end{definition}

Define the weight $w_{c}(M)$ of a matching on $\iota(v,G)$ by
\begin{equation}
  w_{c}(M) = (-1)^{\mathrm{cr}(M)}\prod_{xy,yz\in M}\sqrt{K_{xy}K_{yz}},
\end{equation}
where $\mathrm{cr}(M)$ denotes the number of pairs of edges in the matching
that cross at $v$.

\begin{lemma}
  \label{lem:K2n-Intersection}
  Let $G$ be an even graph, and $v\in V(G)$. Then
  \begin{equation}
    \label{eq:K2n-Intersection.1}
    \sum_{M\in\cc M(\iota(v,G))}w_{c}(M)=\prod_{xy\in V(\iota(v,G))}\sqrt{K_{xy}}
  \end{equation}
\end{lemma}
\begin{proof}
  Choose an edge incident to the vertex $v$ to get a linear order on
  vertices adjacent to $v$ from the cyclic order on edges incident to
  $v$. Define
  \begin{equation*}
    A_{xv,yv} =
    \begin{cases}
      \phantom{-}\sqrt{K_{xv}K_{yv}} & x>y \\
      -\sqrt{K_{xv}K_{yv}} & x<y \\
      0 & x=y
    \end{cases}
  \end{equation*}
  The left-hand side of Equation~\eqref{eq:K2n-Intersection.1} is the
  Pfaffian of the skew symmetric matrix $A$. Using the fact that the
  Pfaffian of $(\xi_{j}\xi_{k}B_{jk})$ equals $\prod \xi_{j}$ times
  the Pfaffian of the matrix $B$ and that the Pfaffian of the
  skew-symmetric matrix that is all 1s above the diagonal is 1 proves
  the claim. Proofs of these facts can be found in~\cite{Stembridge1990}.
\end{proof}

Extend the weight $w_{c}$ on matchings of $\iota(v,G)$ to a weight on
sets of matchings by defining $w_{c} ( \{ M_{v} \}_{v\in V}) \equiv
\prod_{v\in V}w_{c}(M_{v})$. Define $\widetilde{\cc M}(G) \equiv \{
\{M_{v}\}_{v\in V(G)} \mid M_{v}\in\cc M(\iota(v,G))\}$.

\begin{proposition}
  \label{prop:Matching-Decomposition}
  Let $G$ be an even graph. Then
  \begin{equation}
    \prod_{xy\in E(G)}K_{xy} = \sum_{\{M_{v}\}\in\widetilde{\cc M}(G)}w_{c}(\{M_{v}\})
  \end{equation}
\end{proposition}
\begin{proof}
 Let $V=V(G)$. By distributivity
 \begin{align}
   \sum_{\{M_{v}\}\in\widetilde{\cc M}(G)}w_{c}(\{M_{v}\}) &=
   \sum_{\{M_{v}\}\in\widetilde{\cc M}(G)} \prod_{v\in V} w_{c}(M_{v}) \\ 
   &= \prod_{v\in V} \sum_{M_{v}\in \cc M(\iota(v,G))}w_{c}(M_{v}) \\ 
   &= \prod_{v\in V}\prod_{xy\in V(\iota(v,G))}\sqrt{K_{xy}},
 \end{align}
 where the final equality follows from
 Lemma~\ref{lem:K2n-Intersection}.  Observe that each half edge $xy$
 belongs to only the induced subgraphs $\iota(x,G)$ and $\iota(y,G)$,
 so each factor $\sqrt{K_{xy}}$ is contributed exactly twice.
\end{proof}

\subsubsection{Loops and Perfect Matchings}
\label{sec:Loops-Matchings}

A \emph{walk $\gamma$ of length $n$} in a graph $G$ is a sequence
$(\gamma_{1},\ldots, \gamma_{n})$ of vertices $\gamma_{j}\in V(G)$
such that $\gamma_{j}\gamma_{j+1}\in E(G)$. The variable $n$ will
denote the length of a walk $\gamma$ when this meaning is contextually
clear. If $\gamma_{1}=a$ and $\gamma_{n}=b$ then $\gamma$ is a
\emph{walk from $a$ to $b$ in $G$}. The \emph{number of steps}
$\abs{\gamma}$ of a length $n$ walk $\gamma$ is $n-1$.  A walk is
\emph{non-backtracking} if $\gamma_{j}\neq \gamma_{j+2}$ for $1\leq
j\leq \abs{\gamma} -1$, is \emph{closed} if $\gamma_{1}=\gamma_{n}$,
and is \emph{edge-simple} if $\{\gamma_{m}, \gamma_{m+1}\} =
\{\gamma_{n}, \gamma_{n+1}\}$ implies $m=n$. A closed walk $\gamma$ is
\emph{tail free} if the walk $(\gamma_{1}, \dots,
\gamma_{n},\gamma_{2})$ is non-backtracking.

Let $\Walks(G,a,b)$ denote the collection of all walks from $a$ to $b$
in $G$. Let $\Walks(G,a)\equiv \Walks(G,a,a)$,
$\Walks(G)\equiv \cup_{a}\Walks(G,a)$, and define $\nbWalks (G,a,b)$,
$\nbWalks (G,a)$, $\tfWalks(G,a)$ and so on similarly, with the
subscript $\mathrm{nb}$ (resp.\ $\mathrm{tf}$) indicating the walks
are non-backtracking (resp.\ non-backtracking and tail free). When the
graph $G$ is contextually clear the $G$ in the notation may be
omitted.

Let $\CSWalks (G,a)$ be the set of closed edge-simple walks from $a$
to $a$ in $G$, and let $\CSWalks(G)\equiv\cup_{a}\CSWalks
(G,a)$. Define two closed edge-simple walks to be equivalent if they
are equivalent as cyclic sequences, or if one is the reversal of the
other. Denote the set of walks from $a$ to $a$ under this equivalence
relation by $\bCSWalks(a)$. A \emph{loop} is an element of
$\bCSWalks$. Loops are orientable subgraphs: choosing and orienting an
edge of a loop uniquely defines a closed walk in the equivalence class
of the loop.

\begin{definition}
  Two subgraphs $H_{1},H_{2}\subset G$ are \emph{edge disjoint}
  if $E(H_{1})\cap E(H_{2})=\emptyset$. Two subgraphs which are not
  edge disjoint are said to \emph{edge intersect}.
\end{definition}

The notion of edge intersection extends to loops, as loops are
subgraphs with additional structure.

\begin{definition}
  Let $G$ be an even graph. A \emph{decomposition} of $G$ is a set $\{C_{j}\}$ of
  edge disjoint loops such that $\cup_{j}E(C_{j})=E(G)$. Let
  $\cc D(G)$ denote the set of decompositions of an even graph $G$.
\end{definition}

\begin{lemma}
  \label{lem:Matchings-Decompositions}
  Let $G=(V,E)$ be an even graph. The collection $\cc D(G)$ of
  decompositions of $G$ is in bijective correspondence with
  $\widetilde{\cc M}(G) = \{ \{M_{v}\}_{v\in V(G)} \mid M_{v}\in\cc
  M(\iota(v,G))\} $.
\end{lemma}

The bijection is as follows. Choose an edge $\{xv,vw\}$ in a matching
of $\iota(v,G)$, and pick one of the half-edges, say $xv$. This
orients the edge $\{xv,vw\}$ from $xv$ to $vw$, and specifies another
oriented edge $\{vw,wz\}$ in the matching of $\iota(w,G)$. Continuing
in this manner produces a closed edge-simple walk, and hence a loop. This
process can be repeated until the entire collection of matchings is
resolved into loops. Clearly a collection of loops yields a matching,
and this defines the bijection. See Figures~\ref{fig:CSWalk-Matching} and~\ref{fig:Figure-Eight}.

\begin{figure}[h]
  \centering
  \beginpgfgraphicnamed{Figure4}
  \begin{tikzpicture}[scale=2.5]
    \draw[step=.5cm,gray,thin] (-2.5,-.5) grid (-1.5,.5);
    \node[ivertex] (i01) at (-2,0.25) {$\gamma_{3}$};
    \node[ivertex] (i10) at (-1.75,0.0) {$\gamma_{6}$};
    \node[ivertex] (i-10) at (-2.25,0.0) {$\gamma_{7}$};
    \node[ivertex] (i0-1) at (-2.0,-0.25) {$\gamma_{2}$};
    \node[ivertex] (i-1-2) at (-2.25,-.5) {$\gamma_{1}$};
    \node[ivertex] (i-1-1) at (-2.5,-.25) {$\gamma_{8}$};
    \node[ivertex] (i11) at (-1.5,.25) {$\gamma_{5}$};
    \node[ivertex] (i12) at (-1.75,.5) {$\gamma_{4}$};

    \draw[->, thick] (-1.25,0) -- (-.75,0);
    \node at (-1.0,0) [above,text centered] {$\phi$};
    
    \draw[step=.5cm,gray,thin] (-.5,-.5) grid (.5,.5);
    \node[overtex] (o01) at (0.0,0.25) {};
    \node[overtex] (o10) at (0.25,0.0) {};
    \node[overtex] (o-10) at (-0.25,0.0) {};
    \node[overtex] (o0-1) at (0.0,-0.25) {};
    \node[overtex] (o11) at (0.5,0.25) {};
    \node[overtex] (o-1-1) at (-.5,-.25) {};
    \node[overtex] (o12) at (.25,.5) {};
    \node[overtex] (o-1-2) at (-.25,-.5) {};
    \draw[thick,white] [-] (o0-1) -- (o01);
    \draw[thick,white] [-] (o10) -- (o-10);
    \draw[lgraph] [-] (o-1-2) -- (o0-1);
    \draw[sgraph] [-] (o0-1) -- (o01);
    \draw[lgraph] [-] (o01) -- (o12);
    \draw[lgraph] [-] (o12) -- (o11);
    \draw[lgraph] [-] (o11) -- (o10);
    \draw[sgraph] [-] (o10) -- (o-10);
    \draw[lgraph] [-] (o-10) -- (o-1-1);
    \draw[lgraph] [-] (o-1-1) -- (o-1-2);
  \end{tikzpicture}
  \endpgfgraphicnamed
  \caption{The left-hand side of the figure indicates a walk
    $\gamma\in\CSWalks$. The right hand side is the collection of
    perfect matchings $\phi(\gamma)$ associated to $\gamma$. In the
    matchings the dashed edges belong to a copy of $K_{4}$, while the
    dotted edges are matchings on separate copies of $K_{2}$;
    $K_{n}$ denotes the complete graph on $n$ vertices.}
  \label{fig:CSWalk-Matching}
\end{figure}
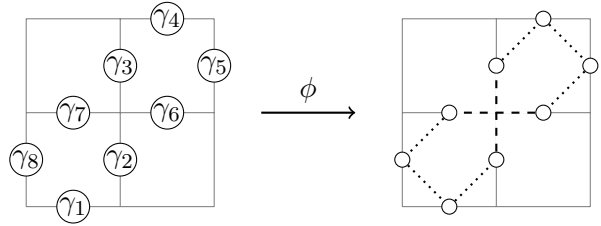

Let $\phi$ be the bijection given by
Lemma~\ref{lem:Matchings-Decompositions}. Define a weight $w$ on
decompositions by $w(\{C_{i}\}) \equiv w_{c} (\phi({\{C_{i}\}}))$, and
define $\mathrm{cr} (\{C_{i}\})$ to be the number of crossings in the
collection of matchings $\phi(\{C_{i}\})$. Representing matchings as
decompositions results in the following distinction between crossings
contained in a single loop and crossings between distinct loops.

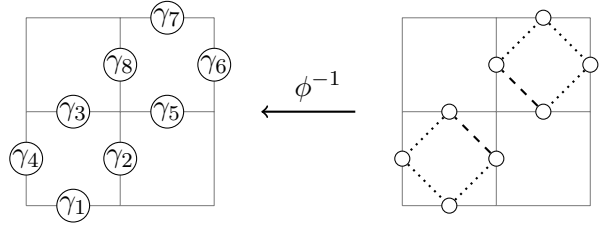
\begin{figure}[h]
  \centering
  \beginpgfgraphicnamed{Figure5}
  \begin{tikzpicture}[scale=2.5]
    \draw[step=.5cm,gray,thin] (-2.5,-.5) grid (-1.5,.5);
    \node[ivertex] (i01) at (-2,0.25) {$\gamma_{8}$};
    \node[ivertex] (i10) at (-1.75,0.0) {$\gamma_{5}$};
    \node[ivertex] (i-10) at (-2.25,0.0) {$\gamma_{3}$};
    \node[ivertex] (i0-1) at (-2.0,-0.25) {$\gamma_{2}$};
    \node[ivertex] (i-1-2) at (-2.25,-.5) {$\gamma_{1}$};
    \node[ivertex] (i-1-1) at (-2.5,-.25) {$\gamma_{4}$};
    \node[ivertex] (i11) at (-1.5,.25) {$\gamma_{6}$};
    \node[ivertex] (i12) at (-1.75,.5) {$\gamma_{7}$};

    \draw[<-, thick] (-1.25,0) -- (-.75,0);
    \node at (-1.0,0) [above,text centered] {$\,\,\,\,\phi^{-1}$};
    
    \draw[step=.5cm,gray,thin] (-.5,-.5) grid (.5,.5);
    \node[overtex] (o01) at (0.0,0.25) {};
    \node[overtex] (o10) at (0.25,0.0) {};
    \node[overtex] (o-10) at (-0.25,0.0) {};
    \node[overtex] (o0-1) at (0.0,-0.25) {};
    \node[overtex] (o11) at (0.5,0.25) {};
    \node[overtex] (o-1-1) at (-.5,-.25) {};
    \node[overtex] (o12) at (.25,.5) {};
    \node[overtex] (o-1-2) at (-.25,-.5) {};
    \draw[lgraph] [-] (o-1-2) -- (o0-1);
    \draw[sgraph] [-] (o0-1) -- (o-10);
    \draw[lgraph] [-] (o01) -- (o12);
    \draw[lgraph] [-] (o12) -- (o11);
    \draw[lgraph] [-] (o11) -- (o10);
    \draw[sgraph] [-] (o10) -- (o01);
    \draw[lgraph] [-] (o-10) -- (o-1-1);
    \draw[lgraph] [-] (o-1-1) -- (o-1-2);
  \end{tikzpicture}
  \endpgfgraphicnamed
  \caption{The right-hand side of the figure indicates perfect
    matchings of a collection of induced subgraphs. The dashed edges
    belong to a copy of $K_{4}$ while the dotted edges belong to
    distinct copies of $K_{2}$. The left hand side is the
    corresponding decomposition.}
  \label{fig:Figure-Eight}
\end{figure}

\begin{definition}
  Let $C_{1},C_{2}\in\bCSWalks$, with $C_{1}$ and $C_{2}$ edge
  disjoint. The \emph{self intersection} $\si{C_{1}}$ of $C_{1}$ is
  given by
  \begin{equation}
    \si{C_{1}}\equiv \mathrm{cr}(C_{1}).
  \end{equation}
  The \emph{mutual intersection} $\mi{C_{1}}{C_{2}}$ of edge disjoint loops $C_{1}$
  and $C_{2}$ is
  \begin{equation}
    \mi{C_{1}}{C_{2}} \equiv \mathrm{cr}(\{C_{1},C_{2}\}) -
    \mathrm{cr} (C_{1}) - \mathrm{cr} (C_{2}).
  \end{equation}
\end{definition}

A crossing only involves two edges, and hence if $\{C_{j}\}\in \cc
D(G)$ is a decomposition
\begin{equation}
  \label{eq:All-Intersections}
    \mathrm{cr}(\{C_{j}\}) = \sum_{j<k} \mi{C_{j}}{C_{k}} + \sum_{j}\si{C_{j}}.
\end{equation}

Equation~\eqref{eq:All-Intersections} allows the weight on
decompositions to be rewritten in a more geometrically intuitive form:
\begin{lemma}
  \label{lem:Decomposition-Weight}
  The weight on decompositions can be written as
  \begin{equation}
  \label{eq:Decomposition-Weight1-Star}
  w(\{C_{j}\}) =
  \prod_{j}(-1)^{\si{C_{j}}} \prod_{j<k}(-1)^{\mi{C_{j}}{C_{k}}}
  \prod_{xy\in \cup E(C_{j})}K_{xy}.
\end{equation}
\end{lemma}

\subsection{Factorization of the Loop Weights}
\label{sec:Loop-Weights}

The representation of the weight on decompositions given by
Equation~\eqref{eq:Decomposition-Weight1-Star} is not suited to
combinatorial analysis, as the weight does not factor as a product of
weights of the individual loops. This situation is remedied by
Lemma~\ref{lem:Surface-Factorization} and
Corollary~\ref{cor:Decomposition-Weight-Prime} which provide a
factorization. We begin by presenting some preliminary definitions
about curves in surfaces.  Lemma~\ref{lem:Surface-Factorization} is
significantly simpler in the planar case, and to help orient the
reader the planar result,
Proposition~\ref{prop:Weight-Factorization-Planar}, is presented first.

\subsubsection{Preliminaries on Curves}
\label{sec:preliminaries-curves}

Recall that the background graph $G$ is always assumed to be properly
embedded in a oriented surface $\surface$ carrying a Riemannian
metric.

\begin{definition}
  A \emph{regular closed curve} in $\surface$ is a smooth map
  $\gamma\colon \cb{a,b}\to\surface$ such that
  \begin{enumerate}
  \item $\gamma(a)=\gamma(b)$,
  \item $\gamma^{\prime}(a)=\gamma^{\prime}(b)$,
  \item$\gamma^{\prime}(t)\neq 0$ for $a\leq t\leq  b$.
  \end{enumerate}
\end{definition}

Regular closed curves will frequently be called regular curves for the
sake of brevity.

\begin{definition}
  Let $\Delta_{2}(\gamma)=\{x\in\surface \mid \textrm{there exist
    exactly two times $t_{1}$,$t_{2}$ such that $\gamma(t_{i})=x$}\}$.
  $\Delta_{2}(\gamma)$ is the set of \emph{double points} of
  $\gamma$.  For $x\in\Delta_{2}(\gamma)$, let $t_{x}^{1}<t_{x}^{2}$
  denote the first and second visits to the double point $x$.
\end{definition}

Points $x\in \Delta_{2}(\gamma)$ for $\gamma$ a regular curve can be classified:
\begin{definition}
  Let $\gamma$ be regular curve, $x\in\Delta_{2}(\gamma)$. Then
  \begin{enumerate}
  \item $x$ is a \emph{positive crossing} if
    $(\gamma^{\prime}(t_{x}^{1}),\gamma^{\prime}(t_{x}^{2}))$ is a
    positively oriented basis for $T_{x}\surface$,
  \item $x$ is a \emph{negative crossing} if
    $(\gamma^{\prime}(t_{x}^{1}),\gamma^{\prime}(t_{x}^{2}))$ is a
    negatively oriented basis for $T_{x}\surface$,
  \item $x$ is a point of \emph{self tangency} if
    $(\gamma^{\prime}(t_{x}^{1}),\gamma^{\prime}(t_{x}^{2}))$ is not a
    basis of $T_{x}\surface$.
  \end{enumerate}
\end{definition}

Throughout this section it will be assumed curves are generic, meaning
all self intersections are double points, and all double points
are either positive or negative crossings.

\subsubsection{Planar Factorization}
\label{sec:planar-factorization}

To describe the idea of factorization in the simplest context, assume
in this section that the graph $G$ embeds properly in
$\R^{2}$. Further, and without loss of generality~\cite{Fary1948},
assume that all edges of the embedded graph are straight lines.

\begin{definition}
  The \emph{turning number} $\tau(\gamma)$ of a closed smooth curve
  $\gamma$ is the net number of rotations made by
  $\gamma^{\prime}(t)/\abs{\gamma(t)}$ around the unit circle as $t$
  ranges from $a$ to $b$.
\end{definition}

Let $N_{\pm}(\gamma)$ denote the number of positive/negative crossings
of a curve $\gamma$. A curve $\gamma$ in $\R^{2}$ is said to be
\emph{supported by the line $L$} if $\gamma(a)\in L$ and $\gamma$ lies
in the closure of a single component of $\R^{2}\setminus L$. A theorem
of Whitney on the double points of plane curves will be used to
exploit the assumption that $G$ is planar.

\begin{theorem}[Whitney~\cite{Whitney1937}]
  \label{thm:Whitney-DPT}
  Let $\gamma$ be a regular closed curve in the upper half plane
  supported by the line $y=0$. Then the signed number of
  self intersections is given by
  \begin{equation}
    \label{eq:1}
    N_{+}(\gamma) - N_{-}(\gamma) = \tau(\gamma) \pm 1
  \end{equation}
  where the sign $+$ is chosen if the horizontal component of
  $\gamma^{\prime}(a)$ is positive, and the sign $-$ is chosen
  otherwise.
\end{theorem}

\begin{definition}
  Let $C$ be a loop, and let $\gamma\in\CSWalks(G)$ be a
  representative member of $C$. The \emph{turning number} of $C$ is
  defined by
   \begin{equation}
     \label{eq:Turn-Def}
    \tau(C) \equiv \frac{1}{2\pi}\sum_{j=1}^{n-1}
    \turn(\gamma_{j},\gamma_{j+2}) \mod 2, \qquad
    \gamma_{n+1}\equiv\gamma_{2}.
  \end{equation}
  where $\turn(\gamma_{j},\gamma_{j+2})$ is the exterior angle of the
  polygonal segment $(\gamma_{j},\gamma_{j+1},\gamma_{j+2})$. See
  Figure~\ref{fig:ExtAngle}.
\end{definition}

\begin{figure}[h]
  \centering
    \beginpgfgraphicnamed{Figure6}
    \begin{tikzpicture}
      \draw (0,1) node[anchor=west] {$x$} -- (0,2) node[anchor=west]
      {$y$}; \draw (0,2) -- (-1.2,2.9) node[anchor=west] {$z$};
      \draw[dashed] (0,2) -- (0,3); \draw[->] (0,2.7) arc (90:135:.76)
      node[anchor=west] {$\,\theta$};
    \end{tikzpicture}
    \endpgfgraphicnamed
    \caption{The turning angle $\theta=\turn(x,z)$ is the exterior
      angle of the polygonal curve consisting of the segments $xy$ and $yz$.}
  \label{fig:ExtAngle}
\end{figure}
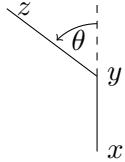

The turning number is well-defined due to the cyclic
nature of the sum and the fact that reversing the direction of the
walk $\gamma$ only changes the sign of the sum, which is irrelevant mod
2. 

\begin{proposition}
  \label{prop:Weight-Factorization-Planar}
  Let $G$ be an even graph properly embedded in the plane. The weight
  $w(\{C_{j}\})$ of a decomposition of $G$ factors over the loops:
    \begin{align}
    \label{eq:1}
    w(\{C_{j}\}) &= \prod_{j}(-1)^{\si{C_{j}}}
    \prod_{j<k}(-1)^{\mi{C_{j}}{C_{k}}} \prod_{xy\in \cup
      E(C_{j})}K_{xy} \\
    &= \prod_{j}\ob{(-1)^{\tau(C_{j})+1}\prod_{xy\in E(C_{j})}K_{xy}}.
  \end{align}
\end{proposition}
\begin{proof}
  The turning number of a loop is equal to the turning number of a
  smoothed version of the piecewise-smooth curve defined by a walk
  representing the loop. Hence Whitney's theorem holds for loops mod
  2. By Lemma~\ref{lem:Crossing-Intersecting} the number of double
  points equals the number of crossings. It remains to show
  $\prod_{j<k}(-1)^{\mi{C_{j}}{C_{k}}}=1$. This follows from the fact
  that $H_{1}(\R^{2},\Z)$ is trivial: the antisymmetry of the
  intersection form implies that two closed loops in the plane must
  intersect an even number of times.
\end{proof}

\subsubsection{The Turning Number on Surfaces}
\label{sec:Turning-Surfaces}

The proof of Proposition~\ref{prop:Weight-Factorization-Planar} used
planarity in two places: once to compute the intersections between two
distinct loops and once in the definition of the turning number. The
first use of planarity is easily generalized by using the
intersection form. To generalize Whitney's formula is less
straightforward. 

Before describing the replacement for Whitney's formula
on a surface, it may be helpful to provide a small amount of
context. In~\cite{Whitney1937}, Whitney characterized the regular
homotopy classes of curves in $\R^{2}$ --- the classes correspond to
the integers $\Z$, and the integer corresponding to a curve is its
turning number. Later, Smale classified the regular homotopy classes
of curves in Riemannian manifolds~\cite{Smale1958}. Smale's
classification did not provide a generalization of the turning number
of Whitney, which spurred several articles searching for such a
quantity~\cite{Reinhart1960,Reinhart1963, Chillingworth1972}.

The turning number in the plane relies on the existence of a canonical
horizontal direction. The initial generalizations of the turning
number compared the rotation of the tangent to a curve with that of a
chosen vector field on the surface. In~\cite{CairnsMcIntyre1993}
Cairns and McIntyre showed that the turning number defined in terms of
a vector field could in fact be computed without choosing a vector
field, provided one is only interested in the turning number mod
$\chi(\surface)$, where $\chi(\surface)$ is the Euler characteristic
of the surface. In what follows this approach will be taken as the
definition of the turning number.

For a surface $\surface$ of genus $g$, let $\{e^{k}_{i}\mid
i\in\{1,2\}, k\in\{1,\ldots,g\}\}$ be a set of smooth curves
representing a symplectic basis of $H_{1}(\surface,\Z_{2})$. That is,
the curves $e^{k}_{i}$ represent generators of
$H_{1}(\surface,\Z_{2})$ such that $\ab{e^{k}_{i},e^{\ell}_{j}} =
\delta_{\ell k}(1-\delta_{ij})$, where $\ab{\cdot,\cdot}$ is the
bilinear intersection pairing on $H_{1}(\surface,\Z_{2})$. For details
on the intersection pairing see~\cite[Chapter 3]{FarkasKra1980}. For
this paper all that is needed is that the intersection pairing counts
the number of points at which two curves in general position intersect
modulo 2.

Let $\gamma$ be a generic regular curve, and let
$C(\gamma)$ denote the set of connected components of $\surface$ after
$\gamma$ and the curves $e^{k}_{i}$, $1\leq k\leq g$, $1\leq i\leq 2$,
are removed. The elements of $C(\gamma)$ are called the
\emph{regions} defined by $\gamma$.

\begin{lemma}[name=Lemma 2 of~\cite{CairnsMcIntyre1993}]
  \label{lem:labelling}
  Let $\gamma$ be a closed regular curve, $\gamma =
  \sum_{k=1}^{g}\sum_{i=1}^{2}n^{k}_{i}e^{k}_{i}$ in homology, and
  suppose $x$ is a point contained in one of the regions $C\in
  C(\gamma)$. There is a labelling by integers of the regions defined
  by $\gamma$ such that the label of a region to the left of $\gamma$
  is 1 greater than the label of the region to the right of $\gamma$,
  and such that the label of a region to the left of $e^{k}_{i}$ is
  $n^{k}_{i}$ less than the label of the region to the right of
  $e^{k}_{i}$. Moreover, if the region containing $x$ is labelled
  zero, this labelling is unique.
\end{lemma}
\begin{proof}
  See~\cite{CairnsMcIntyre1993}.
\end{proof}

In what follows let $\ell_{x}\colon C(\gamma)\to \Z$ be the
unique labelling of regions $C\in C(\gamma)$ such that the region
containing $x$ is labelled $0$.

\begin{definition}
  Assume $x$ is not in the image of $\gamma$. Define $A_{j}$ by
  \begin{equation*}
    A_{j} = \mathrm{cl}\ob{\bigcup_{C\colon\ell_{x}(C)= j}C},
  \end{equation*}
  where $\mathrm{cl}(X)$ denotes the closure of $X$. The \emph{turning
    number} of $\gamma$ is defined to be
  \begin{equation}
    \label{eq:Turning-Definition}
    \tau_{x}(\gamma) =
    \begin{cases}
      \sum_{j\in\Z}j\chi(A_{j}) \mod \abs{\chi(\surface)}, & \surface
      \neq \bb T^{2} \\
      \sum_{j\in\Z}j\chi(A_{j}),  & \surface = \bb T^{2}
    \end{cases}.
  \end{equation}
  where $\chi(X)$ denotes the Euler characteristic of $X$.
\end{definition}

While it is not obvious, this definition does replicate the turning
number of a curve in the plane modulo 2.  An immediate consequence
of~\Cref{thm:General-Intersection-Theorem} below is that the turning
number modulo two is independent of the point $x$ chosen; for this
reason we will omit the subscript $x$ in $\tau_{x}$. 

\subsubsection{Intersections of Curves on the Torus $\bb T^2$}
\label{sec:Torus-Intersections}

For brevity, in this section let $e_{1},e_{2}$ be the standard generators
of $H_{1}(\bb T^{2},\Z_{2})$, i.e., generators which intersect exactly
once. The aim of this section is:

\begin{theorem}
  \label{thm:Intersections-Torus}
  Let $\gamma$ be a regular curve in $\bb T^{2}$. Then
  \begin{equation*}
    N_{+}(\gamma) - N_{-}(\gamma) = \tau(\gamma) + a(\gamma)b(\gamma)
    + a(\gamma) + b(\gamma) + 1 \mod 2,
  \end{equation*}
  where $a(\gamma)=\ab{\gamma,e_{1}}$ and $b(\gamma) = \ab{\gamma,e_{2}}$.
\end{theorem}

The theorem is immediate for $\gamma$ null-homotopic, as this reduces
to Whitney's theorem. For curves that are not null-homotopic, the
proof of the theorem relies on an operation on curves called
\emph{surgery}, which is introduced here in the context of any
surface.

If $\gamma$ is a regular curve, $\gamma
\restr{\cb{a^{\prime},b^{\prime}}}$ will denote the restriction of
$\gamma$ to the interval $\cb{a^{\prime},b^{\prime}}$. Given regular
curves $\gamma_{i}\colon\cb{a_{1},b_{1}}\to \surface$ for
$i\in\{1,2\}$ such that $\gamma_{1}(b_{1})=\gamma_{2}(a_{2})$, let
$\gamma_{1}\gamma_{2}$ denote their concatenation.

\begin{definition}
  Let $\gamma$ be a regular curve and $x\in\Delta_{2}(\gamma)$. The
  result of \emph{surgery at $x$ on $\gamma$} is
  $\{\eta_{1},\eta_{2}\} = \{ \gamma\restr{\cb{0,t_{x}^{1}}}
  \gamma\restr{\cb{t_{x}^{2},1}},\gamma\restr{\cb{t_{x}^{1},t_{x}^{2}}}\}$. See
  Figure~\ref{fig:Surgery-Local}.
\end{definition}

\begin{definition}
  If $\gamma_{1},\gamma_{2}$ are two curves, let $\Delta_{2}
  (\{\gamma_{1}, \gamma_{2}\})$ denote the set of points at which
  $\gamma_{1}$ and $\gamma_{2}$ intersect.
\end{definition}

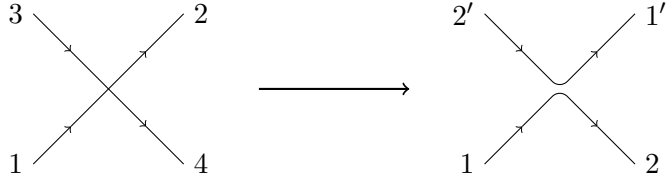
\begin{figure}[h]
  \centering
  \beginpgfgraphicnamed{Figure7}
  \begin{tikzpicture}
    \draw[->] (0,0) --(.5,.5);
    \draw[->] (.5,.5) -- (1.5,1.5);
    \draw (1.5,1.5) -- (2,2);
    \draw[->] (0,2) -- (.5,1.5);
    \draw[->] (.5,1.5) -- (1.5,.5);
    \draw (1.5,.5) -- (2,0);
    \node at (0,0) [anchor=east] {$1$};
    \node at (2,2) [anchor=west] {$2$};
    \node at (0,2) [anchor=east] {$3$};
    \node at (2,0) [anchor=west] {$4$};

    \draw[->,thick] (3,1) -- (5,1);

    \draw[->] (6,0) -- (6.5,.5);
    \draw[->] (6.5,.5) [rounded corners]  --  (7,1) -- (7.5,.5);
    \draw (7.5,.5) -- (8,0);
    \draw[->] (6,2) -- (6.5,1.5);
    \draw[->] (6.5,1.5) [rounded corners] -- (7,1) --(7.5,1.5);
    \draw (7.5,1.5) -- (8,2);
    \node at (6,0) [anchor=east] {$1$};
    \node at (8,0) [anchor=west] {$2$};
    \node at (8,2) [anchor=west] {$1^{\prime}$};
    \node at (6,2) [anchor=east] {$2^{\prime}$};
  \end{tikzpicture}
  \endpgfgraphicnamed
  \caption{The effect of surgery at a double point. The left hand side
    depicts segments of the original curve, with the numbering
    indicating the order the segments are traversed in. The right hand
    side depicts the segments of the curves resulting from surgery. }
  \label{fig:Surgery-Local}
\end{figure}

The following observations about the surgery operation are central to
what follows. Let $\eta^{x}_{1}(\gamma),\eta^{x}_{2}(\gamma)$ be the
curves that result from performing surgery on a regular curve $\gamma$
at the point $x\in\Delta_{2}(\gamma)$.

\begin{lemma}
  \label{lem:Surgery-Points}
  Let $\gamma$ be a generic regular curve, and $x\in\Delta_{2}
  (\gamma)$ be a double point. Then
  \begin{enumerate}
  \item $\Delta_{2}(\gamma) = \{x\}\cup
    \Delta_{2}(\eta^{x}_{1}(\gamma))\cup
    \Delta_{2}(\eta_{2}^{x}(\gamma))\cup \Delta_{2}(\{
    \eta_{1}^{x}(\gamma),\eta_{2}^{x}(\gamma)\})$,
  \item
    $\cb{\gamma}=\cb{\eta_{1}^{x}(\gamma)}+\cb{\eta_{2}^{x}(\gamma)}$
    in $H_{1}(\surface ,\Z_{2})$,
  \item $\tau(\gamma) = \tau(\eta_{1}^{x}(\gamma)) +
    \tau(\eta_{2}^{x}(\gamma))$.
 \end{enumerate}
\end{lemma}
\begin{proof}
  The first two claims are straightforward consequences of the
  definitions. For the third, what is required is that
  \begin{equation}
    \sum_{j}j\chi(A_{j}) = \sum_{j}j\chi(A_{j}^{1})  + \sum_{j}j\chi(A_{j}^{2}),
  \end{equation}
  where $A^{i}_{j}$ is the region labelled $j$ under the labelling due
  to $\eta^{x}_{i}$, and $A_{j}$ the region labelled $j$ under the
  labelling due to $\gamma$. Observe that the labelling from $\gamma$
  gives the sum of the labelings from $\eta_{1}^{x}$ and
  $\eta_{2}^{x}$ as the surgery operation does not change the
  orientation of any segments of the curve. Hence
  \begin{equation}
    \label{eq:Aj-Decomposition}
    A_{j}=\bigcup_{k+\ell=j} A^{1}_{k}\cap A^{2}_{\ell}.
  \end{equation}
  and, letting $K$ denote the Gaussian curvature, 
  \begin{equation}
    \label{eq:Aj-Curvature}
    j\int_{A_{j}}K\, dA = \mathop{\sum_{k,\ell}}_{{k+\ell=j}}\ob{
      k\int_{A^{1}_{k}\cap A^{2}_{\ell}}
      K\, dA + \ell\int_{A^{1}_{k}\cap A^{2}_{\ell}} K\, dA}.
  \end{equation}
  The Gauss-Bonnet formula applied to each region $A_{j}$ yields
  \begin{equation}
    \sum_{j}j\chi(A_{j}) = \sum_{j} j \cb{\int_{A_{j}}K\,dA +
      \int_{\partial A_{j}}k_{g}\,ds},
  \end{equation}
  where $k_{g}$ is the geodesic curvature. Hence it suffices to show
  \begin{equation}
    \label{eq:Curve-Segments}
    \sum_{j}j\int_{\partial A_{j}}k_{g}\,ds =
    \sum_{j}j\sum_{i=1}^{2}\int_{\partial A^{i}_{j}}k_{g}\,ds.
  \end{equation}

  A given segment of a curve appears in the sum over geodesic
  curvatures twice, with opposite orientation. As the labels increase
  when crossing curves to the right, the equality of the contributions
  from integrals along $\gamma$ in \eqref{eq:Curve-Segments} is
  equivalent to
  \begin{equation*}
    \int_{\gamma}k_{g}\,ds =
    \int_{\eta^{x}_{1}}k_{g}\,ds + \int_{\eta^{x}_{2}}k_{g}\,ds,
  \end{equation*}
  which holds as surgery does not modify the orientation or shape of
  the curve outside of a neighbourhood of $x$. A local calculation
  near $x$ shows that the change in geodesic curvature due to surgery
  is zero; see Figure~\ref{fig:Surgery-Local}. Similar reasoning,
  using that the labels increase by $n^{k}_{i}$ across the generator
  $e^{k}_{i}$, shows that the contributions
  to~\eqref{eq:Curve-Segments} from integrals along the generators
  are also equal.
\end{proof}

\begin{remark}
  \label{rem:Turning-Geo-Curvature}
  If the generating curves $e^{k}_{i}$ are chosen to be geodesics, the
  proof of~\Cref{lem:Surgery-Points} shows that the turning
  number of a curve $\gamma$ is the integrated geodesic curvature of
  $\gamma$, along with an additional term involving the total
  curvature of the regions of $C(\gamma)$.
\end{remark}

In the case of the torus $\bb T^{2}$, the turning number is given by

\begin{lemma}[name=Theorem~1 of~\cite{Reinhart1959}]
  \label{lem:Torus-Simple}
  Let $\gamma$ be a simple regular curve in $\bb T^{2}$. Then
  \begin{equation*}
    \tau(\gamma) = 
    \begin{cases}
      \pm 1 & \gamma \textrm{ nullhomotopic}\\
      0 & \gamma\textrm{ homotopically nontrivial}
    \end{cases}.
  \end{equation*}
\end{lemma}
\begin{proof}
  When $\gamma$ is nullhomotopic the claim follows from
  Theorem~\ref{thm:Whitney-DPT}. Suppose $\gamma$ is homotopically
  nontrivial, let $\tilde\gamma$ be a lift of $\gamma$ to
  $\R^{2}$, and for convenience assume $\gamma$ has period one.

  Without loss of generality it can be assumed that the torus carries a
  flat metric. By Remark~\ref{rem:Turning-Geo-Curvature} the turning
  number of $\gamma$ is equivalent to the integrated geodesic
  curvature of $\gamma$ for a surface with a flat metric, and so it
  suffices to compute the integrated geodesic curvature of a period of
  the lift $\tilde\gamma$.

  As $\gamma$ is regular, $\tilde\gamma$ is regular. By changing
  coordinates if necessary, it can be assumed that $\tilde\gamma(0)$
  and $\tilde\gamma(1)$ lie on the $x$-axis. There is then a line of
  support $L$ to $\tilde\gamma\restr{\cb{0,1}}$ parallel to the $x$
  axis; define $0<t^{\star}<1$ to be a point such that
  $\tilde\gamma(t^{\star})$ is on $L$, with
  $\tilde\gamma^{\prime}(t^{\star})$ parallel to $L$. As translation
  by $\tilde\gamma(1)-\tilde\gamma(0)$ is an automorphism of the
  cover, it follow that $L$ is a line of support to
  $\tilde\gamma\restr{\cb{t^{\star},t^{\star}+1}}$. Let $R$ be the
  segment of $L$ with endpoints $\tilde\gamma(t^{\star})$,
  $\tilde\gamma(t^{\star}+1)$ translated vertically by $(0,1)$. See
  Figure~\ref{fig:lem:Torus-Simple}.

  Let $\eta$ be the curve with segments
  $\tilde\gamma\restr{\cb{t^{\star},t^{\star}+1}}$, the vertical
  segments from the endpoints of this curve to $R$, and $R$
  itself. Then, orienting this curve to agree with
  $\tilde\gamma^{\prime}(t^{\star}+1)$ yields $\tau(\eta) =
  \tau(\tilde\gamma\restr{\cb{t^{\star},t^{\star}+a}})$ plus or minus
  $\frac{1}{2\pi}(4\frac{\pi}{2})$, with $+$ if the curve is oriented
  counterclockwise and $-$ otherwise. As $\eta$ is simple by
  construction, Theorem~\ref{thm:Whitney-DPT} implies $\tau(\eta)=\pm
  1$, which completes the proof.
\end{proof}

\begin{figure}[]
  \centering
  \beginpgfgraphicnamed{Figure8}
  \begin{tikzpicture}
    \draw[thick,>->] [rounded corners] (0,0) -- (.25,0) -- (.5,.5) -- (.25,.75)
    -- (.5,1) -- (1, 1.25) -- (1.25,1.25) -- (1.5,1) -- (1.75, .5) --
    (1.5,.25) -- (1.75, .125) -- (1.875,0) -- (2.25,0);
    \draw[dashed] (-1,1.26) -- (6,1.26);
    \draw[dashed] (1.125,1.76) -- (3.375,1.76);
    \node at (6,1.25) [anchor=west] {$L$};
    \node at (3.5,1.75) [anchor=west] {$R$};
    \draw[thick,->] [rounded corners] (2.25,0) -- +(.25,0) -- +(.5,.5) -- +(.25,.75)
    -- +(.5,1) -- +(1, 1.25) -- +(1.25,1.25) -- +(1.5,1) -- +(1.75, .5) --
    +(1.5,.25) -- +(1.75,.125) -- +(1.875,0) --+(2.25,0);
    \draw[dotted,thick] (1.125,1.25) --(1.125,1.75);
    \draw[dotted,thick] (3.375,1.25) -- (3.375,1.75);
    \node at (0,0) [anchor=north] {$\tilde\gamma(0)$};
    \node at (2.25,0) [anchor=north] {$\tilde\gamma(1)$};
    \node at (4.5,0) [anchor=north] {$\tilde\gamma(2)$};
  \end{tikzpicture}
  \endpgfgraphicnamed
  \caption{Illustration of the construction of the curve $\eta$ used
    in the proof of Lemma~\ref{lem:Torus-Simple}. The curve $\eta$
    consists of the segment of the heavy black line between the dotted
    vertical lines, the dotted vertical lines themselves, and the
    dashed line segment $R$.}
  \label{fig:lem:Torus-Simple}
\end{figure}
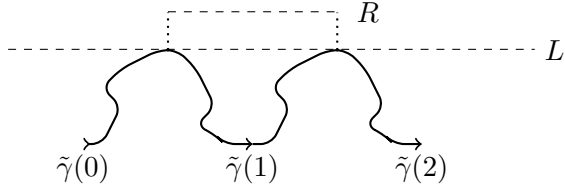

We can now prove~\Cref{thm:Intersections-Torus}.

\begin{proof}[Proof of~\Cref{thm:Intersections-Torus}]
  The proof utilizes Lemma~\ref{lem:Surgery-Points} and is by
  induction on the cardinality of the set of double points. The base
  case of $\Delta_{2}(\gamma)=\emptyset$ follows from
  Lemma~\ref{lem:Torus-Simple} along with the observation that if
  $(a,b)\neq(0,0)$ mod $(2,2)$, then $ab+a+b=1$ mod $2$.

  Suppose the theorem holds for $\abs{\Delta_{2}(\gamma)}=k$, and let
  the result of surgery at $x\in\Delta_{2}(\gamma)$ be the curves
  $\eta_{1}$, $\eta_{2}$. As in the statement of the theorem, let
  $a(\gamma)=\ab{\gamma, e_{1}}$, $b(\gamma)=\ab{\gamma,
    e_{2}}$. Then
  \begin{align*}
    N_{+}(\gamma) - N_{-}(\gamma) &=  \sum_{i=1}^{2} (N_{+}(\eta_{i}) -
    N_{-}(\eta_{i}) ) + \ab{\eta_{1},\eta_{2}} + 1 \mod 2\\
      &= \sum_{i=1}^{2} (\tau(\eta_{i}) +
      a(\eta_{i})b(\eta_{i})+a(\eta_{i}) + b(\eta_{i})) +
      a(\eta_{1})b(\eta_{2}) + a(\eta_{2})b(\eta_{1}) + 1 \mod 2\\
      &= \sum_{j=1}^{2}\tau(\eta_{j}) + 
      (a(\eta_{1})+a(\eta_{2}))(b(\eta_{1})+b(\eta_{2})) +
      \sum_{j=1}^{2}a(\eta_{j}) + \sum_{j=1}^{2}b(\eta_{j}) + 1 \mod 2\\
      &= \tau(\gamma) + a(\gamma)b(\gamma) + a(\gamma) + b(\gamma) +
      1, \mod 2
  \end{align*}
  where the first line is by part~1 of Lemma~\ref{lem:Surgery-Points}, the
  second is by the induction hypothesis and definition of the
  intersection form, the third is a rearrangement,
  and the fourth line uses parts~2
  and~3 of Lemma~\ref{lem:Surgery-Points}. 
\end{proof}

\subsubsection{General Surfaces}
\label{sec:Gen-Surface}

This section presents a formula for the number of intersections of a
regular closed curve on $\surface$ modulo 2. The result is
essentially~\cite[Proposition 4]{Reinhart1963}.

Suppose $\surface$ is a genus $g$ surface. By the classification of
closed surfaces there are $g$ disjoint embedded closed curves
$C_{i}\subset\surface$ such that $\surface = \ob{\coprod_{i=1}^{g} \bb
  T^{2}_{i}}\coprod \cc S^{2}_{g}$, where $\bb T^{2}_{i}$ is a
torus with a disk removed, $\cc S^{2}_{g}$ is the 2-sphere with
$g$ disjoint disks removed, and $\coprod$ indicates disjoint union.

The strategy to find a formula for the number of self intersections
of a curve on a genus $g$ surface is to decompose the curve into
several curves, each of which lies entirely in one of $\bb
T^{2}_{1},\ldots,\bb T^{2}_{g}$ or $\cc S^{2}_{g}$. Analyzing this
procedure will yield
\begin{theorem}
  \label{thm:General-Intersection-Theorem}
  Let $\gamma$ be a regular closed curve on $\surface$. Then
  \begin{equation}
    \label{eq:General-Formula}
    N_{+}(\gamma) - N_{-}(\gamma) = \tau(\gamma) +
    \sum_{i=1}^{g}\big((a_{i}(\gamma)b_{i}(\gamma) + a_{i}(\gamma) +
    b_{i}(\gamma)\big) + 1 \mod 2,
  \end{equation}
  where $a_{i}(\gamma) = \ab{e^{i}_{1},\gamma}$, $b_{i}(\gamma) =
  \ab{e^{i}_{2},\gamma}$. 
\end{theorem}

\begin{remark}
  \label{rem:Turning-Independence}
  The left hand side and homology terms in~\eqref{eq:General-Formula}
  are independent of the point $x$ used to define the turning number.
  It follows that the turning number modulo two is independent of the
  basepoint and orientation of a closed curve.
\end{remark}

\begin{proof}[Proof of~\Cref{thm:General-Intersection-Theorem}]
  Consider the sequence of entrances and exits of a subsurface $\bb
  T^{2}_{i}$ by the curve $\gamma$. There is a cyclic order on these
  events as $C_{i}$ is a circle. Choose a pair of neighbouring events,
  one an entrance, one an exit. Locally the entrance and exit look like
  two straight lines; pulling one of these curves across the other
  creates an intersection inside $\bb T^{2}_{i}$ and an intersection
  outside $\bb T^{2}_{i}$. Performing surgery at the intersection
  inside $\bb T^{2}_{i}$ results in a curve contained in $\bb T^{2}_{i}$,
  and the remainder of $\gamma$ can be smoothly deformed to lie on the
  complement of $\bb T^{2}_{i}$.

  Iterating the above procedure yields a collection
  $\{\gamma_{j}\}_{j=1}^{r+1}$ of closed curves, each contained in a
  component of $\surface\setminus\{C_{i}\}_{i=1}^{g}$, and hence
  \begin{equation*}
    N_{+}(\gamma) - N_{-}(\gamma) = \sum_{j=1}^{r+1}\ob{N_{+}(\gamma_{j})
    - N_{-}(\gamma_{j})} +\sum_{j<k}\ab{\gamma_{j},\gamma_{k}} +  r \mod 2.
  \end{equation*}

  Assume $\gamma_{r+1}$ is the curve remaining on $\cc S^{2}_{g}$. To
  compute $N_{+}(\gamma_{j})-N_{-}(\gamma_{j})$ for $1\leq j\leq r$
  consider the formula that results from viewing $\gamma_{j}$ to be a
  curve on a torus, as filling in the disk removed from the torus does
  not alter the set of intersections. On a torus a curve bounding the
  cut out disk has Euler characteristic $\pm 1$, while on $\surface$ a
  curve bounding the disk has Euler characteristic $\abs{2-2(g-1)-1}$.
  These two numbers agree mod $2$ so
  \begin{equation}
    N_{+}(\gamma_{j})-N_{-}(\gamma_{j}) = \tau(\gamma_{j}) + \sum_{i=1}^{g}\ob{
    a_{i}(\gamma_{j})b_{i}(\gamma_{j}) + a_{i}(\gamma_{j}) +
    b_{i}(\gamma_{j})} + 1 \mod 2,
  \end{equation}
  where the turning number is computed on $\surface$. Similar
  reasoning shows that the same formula holds for the curve
  $\gamma_{r+1}$. Summing these formulas and arguing as in the proof
  of~\Cref{thm:Intersections-Torus} gives the result:
  \begin{align*}
    N_{+}(\gamma)-N_{-}(\gamma) &= \sum_{j=1}^{r+1}N_{+}(\gamma_{j}) -
    N_{-}(\gamma_{j}) + \sum_{j<k}\ab{\gamma_{j},\gamma_{k}} + r\mod 2 \\
    &= \sum_{j=1}^{r+1}\ob{\tau(\gamma_{j}) + \sum_{i=1}^{g}
      \ob{a_{i}(\gamma_{j})b_{i}(\gamma_{j}) + a_{i}(\gamma_{j}) +
      b_{i}(\gamma_{j})} + 1} + \sum_{i=1}^{g} \sum_{j\neq k}
    a_{i}(\gamma_{j}) b_{i}(\gamma_{k}) + r\mod 2 \\
    &= \sum_{j=1}^{r+1}\tau(\gamma_{j}) + \sum_{i=1}^{g}
    \ob{\sum_{j=1}^{r+1}a_{i}(\gamma_{j})\sum_{j=1}^{r+1}b_{i}(\gamma_{j})
    +\sum_{j=1}^{r+1}a_{i}(\gamma_{j}) +
    \sum_{j=1}^{r+1}b_{i}(\gamma_{j})} + 1 \mod 2\\
    &= \tau(\gamma) + \sum_{i=1}^{g}\ob{a_{i}(\gamma)b_{i}(\gamma) +
    a_{i}(\gamma) + b_{i}(\gamma)} + 1 \mod 2.\qedhere
  \end{align*}
\end{proof}

\subsubsection{Factorization on Surfaces}
\label{sec:fact-surf}

In this section we apply the previous results to count the number of
intersections, mod $2$, of a family of curves on a surface
$\surface$. Recall that $\Delta_{2}(\{\gamma_{1},\gamma_{2}\})$ is the
set of points in common between curves $\gamma_{1}$ and
$\gamma_{2}$. We will need the following property of the intersection
form, see~\cite[Chapter 3]{FarkasKra1980} for details.
\begin{lemma}
  \label{lem:Intersection-Form}
  Let $\gamma_{1},\dots,\gamma_{k}$ be smooth curves in general
  position. Then
  \begin{equation}
    \label{eq:Intersection-Form}
    \sum_{j<k}\abs{\Delta_{2}(\{\gamma_{j},\gamma_{k}\})} =
    \sum_{j<k}\ab{\gamma_{j},\gamma_{k}} \mod 2.
\end{equation}
\end{lemma}

Lemma~\ref{lem:Intersection-Form} and a calculation using the fact
that the intersection form is skew symmetric and bilinear yields
\begin{equation*}
  \sum_{j<k}\abs{\Delta_{2}(\{\gamma_{j},\gamma_{k}\})} =
  \sum_{m=1}^{g} \ob{A_{m}B_{m}  -
    \sum_{j}a_{m}(\gamma_{j})b_{m}(\gamma_{j})} \mod 2,
\end{equation*}
where $A_{m}\equiv \sum_{j}a_{m}(\gamma_{j})$, $B_{m}\equiv
\sum_{j}b_{m}(\gamma_{j})$. The self intersections of a
curve are not counted by the intersection form, but
Theorem~\ref{thm:General-Intersection-Theorem} showed that
\begin{equation*}
  \sum_{j}\abs{\Delta_{2}(\gamma_{j})} =
  \sum_{j} \ob{\tau(\gamma_{j})+1} + \sum_{m=1}^{g}\ob{A_{m}+B_{m} +
  \sum_{j}a_{m}(\gamma_{j})b_{m}(\gamma_{j})} \mod 2,
\end{equation*}
and hence
\begin{equation}
  \label{eq:Surface-Factorization-1}
  \prod_{j} (-1)^{\abs{\Delta_{2}(\gamma_{j})}} \prod_{j<k}
  (-1)^{\abs{\Delta_{2}(\{\gamma_{j},\gamma_{k}\})}} =
  (-1)^{\sum_{j}\ob{\tau(\gamma_{j})+1} + \sum_{m=1}^{g}\ob{A_{m}B_{m}+ A_{m}+B_{m}}}.
\end{equation}
Define, for $\alpha\in H_{1}(\surface,\Z_{2})$,
\begin{equation}
  \label{eq:Constant-c}
  (-1)^{c(\alpha)}=\prod_{k=1}^{g}(-1)^{1+\ab{\alpha,e^{k}_{1}}+\ab{\alpha,e^{k}_{2}}
    + \ab{\alpha,e^{k}_{1}}\ab{\alpha,e^{k}_{2}}}. 
\end{equation}

\begin{lemma}
  \label{lem:Surface-Factorization}
  Let $\{\gamma_{j}\}$ be a collection of generic closed curves.  Then
  \begin{equation}
    \label{eq:Surface-Factorization}
    \prod_{j} (-1)^{\abs{\Delta_{2}(\gamma_{j})}} \prod_{j<k}
    (-1)^{\abs{\Delta_{2}(\{\gamma_{j},\gamma_{k}\})}} = \frac{1}{2^{g}}
    \sum_{\alpha\in H_{1}(\surface,\Z_{2})} (-1)^{c(\alpha)}
    \prod_{j}(-1)^{\tau(\gamma_{j}) + \ab{\alpha,\gamma_{j}}+1},
  \end{equation}
\end{lemma}
\begin{proof}
  This follows from applying the identity $(-1)^{ab+a+b}=
  \frac{1}{2}(-1 + (-1)^{a} +(-1)^{b} +(-1)^{a+b})$ to the terms
  $A_{m}B_{m}+A_{m}+B_{m}$ in
  Equation~\eqref{eq:Surface-Factorization-1}, expanding
  $\prod_{m=1}^{g}\frac{1}{2}\ob{-1 + (-1)^{A_{m}} + (-1)^{B_{m}}+
    (-1)^{A_{m} + B_{m}}}$, and using the definition of $c(\alpha)$.
\end{proof}

\begin{definition}
  Recall that a loop is an equivalence class of closed edge simple
  walks. The \emph{turning weight} of a loop $C$ is defined to be
  $(-1)^{\tau(\gamma)}$ where $\gamma$ is a smoothed version of a walk
  in the class $C$.
\end{definition}

A smoothed version of a walk means that the corners where edges meet
are rounded off in a smooth fashion, and that any triple (quadruple,
etc.) intersections are perturbed to a collection of double points. As
self and mutual intersections of loops are defined pairwise, this
count of double points via smoothing agrees with the count of double
points for loops. Lemma~\ref{lem:Crossing-Intersecting} thus implies
Lemma~\ref{lem:Surface-Factorization} applies to collections of
edge-disjoint loops, i.e., to decompositions. Recalling
Lemma~\ref{lem:Decomposition-Weight}, this yields
\begin{corollary}
  \label{cor:Decomposition-Weight-Prime}
  The weight of a decomposition $\{C_{j}\}$ can be written
  \begin{equation}
  \label{eq:Decomposition-Weight1}
  w(\{C_{j}\}) = \frac{1}{2^{g}}
    \sum_{\alpha\in H_{1}(\surface,\Z_{2})} (-1)^{c(\alpha)}
    \prod_{j}\ob{(-1)^{\tau(C_{j}) +
      \ab{\alpha,C_{j}}+1}\prod_{xy\in E(C_{j})}K_{xy}}.
\end{equation}
\end{corollary}

\subsection{Totally Nonbacktracking Walks}
\label{sec:TNB-Walks}

The theory of heaps of pieces is used throughout this section. A good
introduction can be found in either
of~\cite{Krattenthaler2000,Viennot1986}. The basic definitions,
terminology, and theorem used in this paper are recalled in
Appendix~\ref{sec:Heaps} and should be consulted now by readers who
are not familiar with the theory.

\subsubsection{Heuristic Idea of the Heap-Walk Correspondence}
\label{sec:TNB-Walks-Heuristic}

This section present a bijection between a class $\tnbWalks$ of walks
and the set of pairs $(\omega^{\star},H)$, where $\omega^{\star}$ is a
walk that does not repeat any edge, and $H$ is a heap of loops. To
help orient the reader a heuristic description is presented first,
beginning with a description of the class $\tnbWalks$. The formal
definition of $\tnbWalks$ is Definition~\ref{def:Acceptable} below.

Trace a non-backtracking walk $\omega$ until the second time $t_{2}$
that a traversal of an oriented edge $xy$ is completed. Let $t_{1}$ be
the first time the oriented edge $xy$ is traversed. Remove the closed
subwalk from $t_{1}-1$ to $t_{2}-1$ from $\omega$. If either (i) the portion
of $\omega$ that remains contains a backtrack or (ii) the
removed subwalk repeats an edge then the walk $\omega$ is discarded.
Otherwise the same procedure is repeated to the portion of $\omega$
that remains after the removal of the closed subwalk. Walks $\omega$ for which
this iterated procedure results in a walk that does not traverse any edge twice
are called \emph{totally non-backtracking}, denoted
$\omega\in\tnbWalks$. See~\Cref{fig:Magnifying-Lens} for an example of
how a walk can fail to be totally non-backtracking.

\begin{figure}
  \centering
  \beginpgfgraphicnamed{Figure9S}
  \begin{tikzpicture}[scale=1.5]
    \draw[step=1cm,gray] (-2,-1) grid (1,2);
    \node[ivertex] (v00) at (-2,-1) {};
    \node[ivertex] (v10) at (-1,-1) {};
    \node[ivertex] (v20) at (0,-1) {};
    \node[ivertex] (v01) at (-2,0) {};
    \node[ivertex] (v11) at (-1,0) {};
    \node[ivertex] (v21) at (0,0) {};
    \node[ivertex] (v12) at (-1,1) {};
    \node[ivertex] (v22) at (0,1) {};
    \node[ivertex] (v32) at (1,1) {};
    \node[ivertex] (v13) at (-1,2) {};
    \node[ivertex] (v23) at (0,2) {};
    \node[ivertex] (v33) at (1,2) {};
    \node at (v00) [above left] {$\omega_{1}$};
    \node at (v10) [above left] {$\omega_{2}$};
    \node at (v11) [below left] {$\omega_{3}$};
    \node at (v01) [below left] {$\omega_{4}$};
    \node at (v00) [below] {$\omega_{5}$};
    \node at (v10) [below] {$\omega_{6}$};
    \node at (v20) [below left] {$\omega_{7}$};
    \node at (v21) [below left] {$\omega_{8}$};
    \node at (v22) [below left] {$\omega_{9}$};
    \node at (v32) [below right] {$\omega_{14}$};
    \node at (v33) [above right] {$\omega_{13}$};
    \node at (v23) [above] {$\omega_{12}$};
    \node at (v13) [above left] {$\omega_{11}$};
    \node at (v12) [below left] {$\omega_{10}$};
    \node at (v22) [below right] {$\omega_{15}$};
    \node at (v21) [below right] {$\omega_{16}$};
    \node at (v20) [below right] {$\omega_{17}$};
    \draw[black, very thick] [rounded corners] (v00) -- (v10) -- (v11) -- (v01) --
    (v00) -- (v10) -- (v20) -- (v21) -- (v22) -- (v32) -- (v33) --
    (v23) -- (v13) -- (v12) -- (v22) -- (v21) -- (v20); 
  \end{tikzpicture}
  \endpgfgraphicnamed
  \caption{The walk $\omega$ illustrated is a closed non-backtracking
    walk, but it is not totally non-backtracking. After removing the
    first loop the remaining walk is not acceptable, as some edges are
    traversed twice.}
  \label{fig:Magnifying-Lens}
\end{figure}

The procedure to identify a totally non-backtracking walk always
removes a closed subwalk from a non-backtracking walk. These closed
subwalks $\{C_{j}\}$ inductively form a partially ordered set (poset)
by defining a closed subwalk $C_{1}$ to be greater than $C_{2}$ if
$C_{1}$ was removed after $C_{2}$, and $C_{1}$ and $C_{2}$ share an
edge. Note that the maximal elements of this poset always share an
edge with the remaining portion of $\omega$. It follows that a totally
non-backtracking walk $\omega$ can be transformed into a poset of
closed subwalks along with a walk $\omega^{\star}$ that does not
traverse any edge twice, and that each maximal closed subwalk shares
an edge with $\omega^{\star}$. If we forget the orientation of each
$C_{i}$, i.e., only remember the equivalence class (loop) to which
$C_{i}$ belongs, this partially ordered structure is a heap of
loops.

Let $(\omega^{\star},H)$ be a pair, $\omega^{\star}$ a walk that does
not traverse any edge twice, and $H$ a heap of loops whose maximal
elements contain an edge in $\omega^{\star}$. To invert the previously
described procedure first trace the walk $\omega^{\star}$. This
specifies an orientation of the loops that are maximal in $H$
by orienting the first edge of the loop that is contained in
$\omega^{\star}$ to have the same orientation as in
$\omega^{\star}$. In this way each maximal loop in the heap can be
identified with a closed walk. Tracing $\omega^{\star}$ also
gives an order on the maximal loops of the heap $H$, since the loop
labelling a maximal piece shares an edge with $\omega^{\star}$. The
order is the reverse of the order in which the shared edges are
encountered. The first loop in this order can be inserted into
$\omega^{\star}$ by inserting the corresponding closed subwalk
at the first shared edge. This defines a new walk $\omega_{1}$, and
$\omega_{1}$ induces a new order on the remaining maximal loops of
$H$. Repeatedly applying this procedure recovers the original
path $\omega$. The next section makes the preceding discussion
precise.

\subsubsection{Bijection between Totally Nonbacktracking Walks and
  Heaps of Loops}
\label{sec:TNB-Walks-Bijection}

In this section it will be important to distinguish between the edges
$\{x,y\}$ of $G$ and oriented edges $(x,y)$. To lighten the notation
we will write $xy$ for an unoriented edge and $\orient{xy}$ for an
oriented edge. The word ``edge'' without any modifier will always
refer to an unoriented edge. A walk $\omega$ is \emph{edge closed} if
the initial and terminal oriented edges coincide, i.e.,
$\orient{\omega_{1}\omega_{2}} = \orient{\omega_{n-1}\omega_{n}}$. The
next definition extends the notion of being edge-simple to edge closed walks.

\begin{definition}
  \label{def:Edge-Simple}
  An edge closed walk is edge-simple if
  $\omega_{i}\omega_{i+1}=\omega_{j}\omega_{j+1}$ implies either
  $i=j$, or $i=1$, $j=n-1$, or vice versa.
\end{definition}

If $\omega$ is an $n$ step edge closed walk, the \emph{shift} of
$\omega$ is the walk $(\omega_{2}, \dots, \omega_{n},
\omega_{3})$. The \emph{reversal} of $\omega$ is $(\omega_{n-1},
\omega_{n-2}, \dots, \omega_{1},
\omega_{n-2})$. See~\Cref{fig:Shift-Reverse}.  The definitions just
given imply the next lemma.

\begin{figure}
  \centering
  \beginpgfgraphicnamed{Figure10}
  \begin{tikzpicture}
    \draw[dashed] (0,0) circle(2);
    \path (315:2) coordinate (a);
    \path (345:2) coordinate (b);
    \path (375:2) coordinate (c);
    \path (405:2) coordinate (d);
    \node[circle,fill=black] at (a) {};
    \node[circle,fill=black] at (b) {};
    \node[circle,fill=black] at (c) {};
    \node[circle,fill=black] at (d) {};
    \node[anchor=west] at (a) {$\,\,\,\omega_{n-2}$};
    \node[anchor=west] at (b) {$\,\,\,\omega_{1}=\omega_{n-1}$};
    \node[anchor=west] at (c) {$\,\,\,\omega_{2}=\omega_{n}$};
    \node[anchor=west] at (d) {$\,\,\,\omega_{3}$};
  \end{tikzpicture}
  \endpgfgraphicnamed
  \caption{A schematic depiction of an $n$-step edge closed walk $\omega$.}
  \label{fig:Shift-Reverse}
\end{figure}
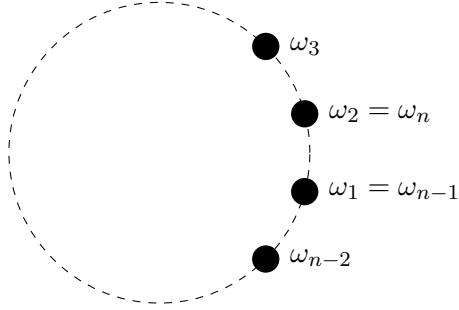

\begin{lemma}
  \label{lem:Shifting-Reversing}
  If $\omega$ is an edge closed non-backtracking walk, the shift and
  reversal of $\omega$ are also edge closed and non-backtracking. If
  $\omega$ is edge-simple, so are the shift and reversal of $\omega$.
\end{lemma}

An \emph{oriented loop} is a edge-simple edge closed walk.
An equivalence class (under shifts and reversals) of oriented loops is
called a \emph{EC-loop}. It is a simple matter to relate EC-loops to
loops.

\begin{lemma}
  \label{lem:Loop-EC-Loop}
  There is a bijection between the set of EC-loops and the set
  of loops.
\end{lemma}
\begin{proof}
  A closed edge-simple walk $\omega$ is a sequence $(\omega_{1}, \dots,
  \omega_{n})$ of adjacent vertices that begins and ends at the same
  vertex such that no unoriented edge is used twice. The oriented loop
  corresponding to $\omega$ is the sequence $(\omega_{1}\omega_{2}, 
  \dots, \omega_{n-1}\omega_{n},\omega_{n}\omega_{2})$ of edges
  traversed, with the first edge repeated. This is a bijection, and
  yields a bijection between loops (equivalence classes of closed
  edge-simple walks) and EC-loops (equivalence classes of oriented loops).
\end{proof}

A \emph{heap of loops} is a heap of pieces whose piece types are
loops, with two loops being concurrent if they have an edge in
common. Deferring the precise definition of a totally non-backtracking
walk until Definition~\ref{def:Acceptable} below, the formal statement
of the heuristic discussion in Section~\ref{sec:TNB-Walks-Heuristic}
is

\begin{theorem}
  \label{thm:Walks-Bijection}
  There is a bijection from the set of totally non-backtracking walks
  to the set of pairs $(\omega^{\star},H)$, where $\omega^{\star}$ is
  an edge-simple walk and $H$ is a heap of loops whose maximal
  elements' labels contain an edge belonging to $\omega^{\star}$.
\end{theorem}

A \emph{heap of EC-loops} is a heap of pieces whose piece types are
EC-loops, with two EC-loops being concurrent if they have an edge in
common. The correspondence between loops and EC-loops given by
Lemma~\ref{lem:Loop-EC-Loop} means that it suffices
to prove Theorem~\ref{thm:Walks-Bijection} when $H$ is a heap of
EC-loops. The use of EC-loops is for notational convenience.

\begin{definition}
  \label{def:Return-Time}
  The \emph{first return time $\tau_{\omega}$} of a non-backtracking
  walk $\omega$ is defined by
  \begin{equation}
    \label{eq:Return-Time} 
    \tau_{\omega} = \min \set{ k \mid \textrm{there exists a $j<k$ such that
      $\orient{\omega_{j}\omega_{j+1}} =
      \orient{\omega_{k}\omega_{k+1}}$}}.
  \end{equation}
  By convention $\min\emptyset=\infty$. If $\tau_{\omega}<\infty$
  let $\tau^{\star}_{\omega}$ denote the unique $j<\tau_{\omega}$ such
  that $\orient{\omega_{j}\omega_{j+1}}=\orient{\omega_{\tau_{\omega}}
    \omega_{\tau_{\omega}+1}}$.
\end{definition}

For $1\leq j \leq k\leq n$ define $\omega^{k}_{j}=(\omega_{j}, \dots,
\omega_{k})$. The definition of the first return time implies that the
first edge-closed subwalk of a walk $\omega$ is
$\omega_{\tau_{\omega}^{\star}}^{\tau_{\omega}+1}$, provided such a
subwalk exists. Denote this edge closed subwalk by $L(\omega)$. If
$\omega$ is a walk of length $n$ define $\omega\setminus\omega^{k}_{j}
= (\omega_{1}, \dots, \omega_{j-1}, \omega_{k+1}, \dots, \omega_{n})$.

\begin{definition}
  \label{def:LE}
  The \emph{(single) loop erasure} of a non-backtracking walk $\omega$
  is the walk $\LE(\omega) = \omega \setminus
  \omega^{\tau_{\omega}}_{\tau_{\omega}^{\star}+1}$.
\end{definition}

\begin{definition}
  \label{def:Acceptable}
  A non-backtracking walk $\omega$ of length $n$ is \emph{acceptable}
  if $(\omega_{1}, \dots, \omega_{\tau_{\omega}\wedge n})$ is an
  edge-simple walk. A non-backtracking walk is \emph{totally
    non-backtracking} if $\LE^{k}(\omega)$ is acceptable for all $k$.
\end{definition}

In other words, a walk $\omega$ is acceptable if the edge closed walk
$L(\omega)$ removed by loop erasure is an oriented loop, and
this oriented loop does not share any \emph{unoriented} edge with the
portion of $\omega$ that occurs prior to $L(\omega)$. Define $\bar
L(\omega)$ to be the EC-loop (equivalence class of oriented loops) to
which $L(\omega)$ belongs.

The following algorithm is one half of the bijection between 
non-backtracking edge closed walks and heaps of EC-loops.

\begin{definition}[Loop erasure algorithm]
  \label{def:LErase}
  Let $\omega$ be a totally non-backtracking walk. The
  \emph{loop erasure} of $\omega$ is the output $(\omega^{\star},H_{\omega})$ of
  the following algorithm.
  \begin{enumerate}
    \item Initialize $(\gamma^{0},H^{0}) = (\omega,\emptyset)$, where
      $\emptyset$ is the empty heap of EC-loops.
    \item
      \begin{enumerate}
      \item If $\gamma^{i-1}$ contains an edge-closed subwalk, set
      $\gamma^{i}=\LE(\gamma^{i-1})$. Set $H^{i}$
      to be the heap of EC-loops $H^{i-1}\circ \bar L(\gamma^{i-1})$.
      \item If $\gamma^{i-1}$ is edge-simple, output $(\omega^{\star},H_{\omega})
        = (\gamma^{i-1},H^{i-1})$.
      \end{enumerate}
    \item Go to 2.
  \end{enumerate}
\end{definition}

The loop erasure of a totally non-backtracking walk is always
acceptable. It follows that the loop erasure algorithm is
well-defined as any walk has finite length and EC-loops have positive
length, so the algorithm terminates after a finite number of steps.

\begin{lemma}
  \label{lem:Erasure-Pyramid}
  Let $\omega$ be a totally non-backtracking walk. The
  pair $(\omega^{\star},H_{\omega})$ output by the loop erasure
  algorithm applied to $\omega$ consists of an edge-simple walk
  $\omega^{\star}$ and a heap of EC-loops whose maximal
  elements' labels contain an edge in $\omega^{\star}$.
\end{lemma}
\begin{proof}
  The walk $\omega^{\star}$ is edge-simple: it cannot contain an edge
  closed loop by the definition of the algorithm, so
  $\tau_{\omega^{\star}}=\infty$. A consequence of the definition of
  an acceptable walk is that $\tau_{\omega^{\star}}=\infty$; this implies
  that $\omega^{\star}$ itself is edge-simple. As each EC-loop removed
  by the algorithm shares an edge with the remaining walk the labels
  of the maximal elements of $H_{\omega}$ contain edges of
  $\omega^{\star}$.
\end{proof}

The inverse of the loop erasure algorithm is given by another
algorithm, the \emph{loop addition algorithm}. The definition requires
a description of how EC-loops can be inserted into a walk.

\begin{definition}
  Let $\omega=(\omega_{1},\dots,\omega_{n})$ be a non-backtracking
  walk, and let $\eta$ be an EC-loop sharing an edge with
  $\omega$. Let $\orient{\omega_{s}\omega_{s+1}}$ be the first
  oriented edge in $\omega$ such that $\omega_{s}\omega_{s+1}$ occurs
  in $\eta$, and let $\tilde\eta = (\eta_{1},\dots, \eta_{m})$ be the
  edge closed walk obtained from $\eta$ by tracing the oriented loop
  given by~\Cref{lem:Loop-EC-Loop} that begins with the oriented edge
  $\orient{\omega_{s}\omega_{s+1}}$. The \emph{loop insertion
    $\omega\loopadd\eta$ of $\eta$ into $\omega$} is defined by
  \begin{equation*}
    \omega\loopadd\eta = (\omega_{1},\dots, \omega_{s}, \eta_{2},
    \dots, \eta_{m-1}, \omega_{s+1},\dots, \omega_{n}).
  \end{equation*}
\end{definition}

Loop erasure creates a heap of EC-loops; defining an inverse algorithm
requires knowing the order in which EC-loops should be inserted back
into a walk. Oriented loops are erased chronologically, so after
removing several oriented loops from a walk the oriented loop which is
\emph{furthest} from the beginning of the remaining walk should be
added first. More precisely,

\begin{definition}
  Let $\omega$ be a non-backtracking walk, and $\cc C = \{C_{1},\dots,
  C_{k}\}$ be a collection of EC-loops. The \emph{walk order induced by
    $\omega$} on $\cc C$ is given by $C_{1}<C_{2}$ if the first
  occurrence of an edge in $C_{1}$ in $\omega$ is earlier than the
  first occurrence of an edge in $C_{2}$ in $\omega$.
\end{definition}

\begin{lemma}
  \label{lem:Strict-Total-Order}
  Let $\omega$ be a non-backtracking walk, and let $H$ be a heap of
  EC-loops. Assume the label of each maximal piece in $H$ contains an edge
  in $\omega$. The walk order on the collection $\{\ell(x) \mid \textrm{$x$
    is a maximal piece in $H$}\}$ is a strict total order.
\end{lemma}
\begin{proof}
  Two pieces that are maximal in the heap order cannot have labels
  that share an edge.
\end{proof}

Define the maximal element of a heap with respect to the walk order to
be the maximal piece $x\in H$ such that $\ell(x)$ is maximal in the
walk order. This is well-defined
by~\Cref{lem:Strict-Total-Order}. When it is necessary to indicate the
particular walk $\eta$ that is being used to order elements in a heap this
will be noted by a subscript, e.g., $\max_{\eta}$.

\begin{definition}
  \label{def:Legal-Pair} A \emph{legal pair} is a pair
  $(\omega^{\star},H)$ where $\omega^{\star}$ is an edge-simple walk and
  $H$ is a heap of EC-loops whose maximal elements' labels contain an
  edge in $\omega^{\star}$.
\end{definition}

\begin{definition}[Loop addition algorithm]
  \label{def:LAdd}
  Let $(\omega^{\star},H)$ be a legal pair. The \emph{loop addition}
  of $(\omega^{\star},H)$ is the walk $\omega$ output by the following
  algorithm:
  \begin{enumerate}
  \item Set $\omega^{0}=\omega^{\star}$, $H^{0}=H$,
  \item
    \begin{enumerate}
    \item If $H^{i-1}\neq\emptyset$, set $\omega^{i}=\omega^{i-1}\loopadd
      \ell( \max_{\omega^{i-1}} H^{i-1})$, where $\ell$ is the label
      function of the heap $H^{i-1}$, and set $H^{i}=H^{i-1}\setminus
      \mathrm{argmax}_{\omega^{i-1}}H^{i-1}$.
    \item if $H^{i-1}=\emptyset$, output $\omega=\omega^{i-1}$.
    \end{enumerate}
  \item Go to 2.
  \end{enumerate}
\end{definition}

By construction the labels of the maximal pieces of $H^{i}$ share an
edge with $\omega^{i}$, so the loop addition algorithm produces a walk
and is well-defined. To prove that the output of the loop addition
algorithm is a totally non-backtracking walk requires a lemma.

\begin{lemma}
  \label{lem:Insert-Remove}
  Let $\omega$ be the output of the loop addition algorithm
  applied to a legal pair $(\omega^{\star},H)$. Assume $\abs{H}=k$, so
  $\omega=\omega^{k}$. Then $\LE(\omega)=\omega^{k-1}$.
\end{lemma}
\begin{proof}
  The proof is by induction on $\abs{H}$. Let $C$ be the EC-loop
  inserted at the final step of the loop addition algorithm. Then $C$
  is the unique maximal EC-loop in $H^{k-1}$. There are two cases
  to consider, depending on if $H^{k-2}$ is a trivial heap or a
  pyramid heap.
  \begin{enumerate}
  \item Suppose $H^{k-2}$ is a trivial heap, i.e., $C$ is the label of
    a maximal element in $H^{k-2}$. In this case $C$ is edge disjoint
    with the EC-loop $C^{\prime}$ added at step $k-1$, and hence
    occurs later in the walk order. This implies $C$ closes prior to
    $C^{\prime}$ in $\omega$, and by induction, that $C$ is the first
    edge closed walk to close in $\omega$.
  \item If $H^{k-2}$ is not a trivial heap then $C$ was not the label
    of a maximal element in $H^{k-2}$. In this case $C$ is inserted
    into the subwalk $C^{\prime}$ of $\omega^{k-1}$, where
    $C^{\prime}$ was the label of the unique maximal element of
    $H^{k-2}$. This implies $C$ closes prior to $C^{\prime}$, and by
    induction $C$ is the first edge closed subwalk to close.\qedhere
  \end{enumerate}
\end{proof}

\begin{lemma}
  \label{lem:Addition-Properties}
  The loop addition algorithm applied to a legal pair
  $(\omega^{\star},H)$ produces a totally non-backtracking walk
  $\omega$ that begins with the edge
  $\orient{\omega^{\star}_{1}\omega^{\star}_{2}}$.
\end{lemma}
\begin{proof}
  Assume by induction that the lemma holds for $\abs{H}=k-1$ for some
  $k\geq 1$. Consider a legal pair with $\abs{H}=k$. By induction the
  walk $\omega^{k-1}$ produced by the first $k-1$ steps of the loop
  addition algorithm has first step
  $\orient{\omega^{\star}_{1}\omega^{\star}_{2}}$. Loop insertion
  never alters the first step of a walk, so
  $\orient{\omega^{\star}_{1}\omega^{\star}_{2}}$ is the first step of
  $\omega=\omega^{k}$. Furthermore, as the last loop is glued in at the
  \emph{first} shared edge it follows that $(\omega^{k}_{1}, \dots,
  \omega^{k}_{\tau_{\omega^{k}}})$ is an edge-simple walk, so
  $\omega^{k}$ is acceptable. \Cref{lem:Insert-Remove} implies
  $\LE(\omega^{k}) = \omega^{k-1}$, and by induction it follows that
  $\omega$ is totally non-backtracking.
\end{proof}

\begin{proof}[Proof of~\Cref{thm:Walks-Bijection}]
  The loop erasure (resp.\ loop addition) algorithm has already
  established that every such walk (resp.\ pair) is mapped to such a
  pair (resp.\ walk). To prove the theorem it suffices to show that
  the loop insertion and loop removal operations add and remove the
  same loops when applied in succession to one another, as this shows
  that the heap built by loop erasure is the same as the heap
  dismantled by loop addition.

  Let $\omega$ be the output of the loop addition algorithm applied to
  a legal pair $(\omega^{\star},H)$. \Cref{lem:Insert-Remove}
  establishes that loop erasure applied to $\omega$ results in
  removing the last EC-loop added by loop erasure. As the walk that
  remains after removing one EC-loop is also the result of loop
  addition, it follows that successive applications of $\LE$
  remove EC-loops from $\omega$ in the inverse of the order they were added.

  Let $\omega$ be a totally non-backtracking walk. The definition of
  an acceptable walk implies the first edge $L(\omega)$ shares with
  $\LE(\omega)$ is $\LE(\omega)_{\tau_{\omega}^{\star}}
  \LE(\omega)_{\tau_{\omega}^{\star}+1}$; the definition of loop
  insertion thus implies that $\LE(\omega)\loopadd L(\omega) =
  \omega$. As the walk that remains after removing a single EC-loop is
  still a totally non-backtracking walk, it follows that the loop
  addition algorithm adds EC-loops in the inverse of the order in
  which they were removed.
\end{proof}

\begin{corollary}
  \label{cor:Walks-Pyramids}
  The set of tail free totally non-backtracking walks with initial
  edge $\orient{xy}$ are in bijective correspondence with the set of
  legal pairs $(\orient{xy},P)$, where $P$ is a pyramid of loops whose
  maximal element's label contains $xy$.
\end{corollary}
\begin{proof}
  Any tail free walk extends to an edge closed walk by repeating the
  first edge of the walk. If an edge closed walk is totally
  non-backtracking all that remains after loop erasure is a single
  edge, which by~\Cref{lem:Addition-Properties} must by
  $\orient{xy}$. Conversely, any such legal pair corresponds to an
  edge closed walk, which corresponds to a tail free walk by removing
  the last edge.
\end{proof}

\subsubsection{An Involution on Non-Backtracking Walks}
\label{sec:Involution}

\Cref{sec:TNB-Walks-Bijection} established a bijection between heaps
of EC-loops and totally non-backtracking walks. This section
constructs an involution on walks that are \emph{not} totally
non-backtracking. The involution will be used
in~\Cref{sec:proof-main-theorem} to show that sums over totally
non-backtracking walks can be expanded to sums over non-backtracking
walks without altering the value of the sum.

If $\omega$ is not acceptable then the subwalk $(\omega_{1}, \dots,
\omega_{\tau_{\omega}})$ is not edge-simple. By construction no
oriented edge is used twice in $(\omega_{1}, \dots,
\omega_{\tau_{\omega}})$. Let $xy$ be the first edge used with both
orientations, and let $i$ and $j$ be the indices of $\omega$ such that
$\orient{\omega_{i}\omega_{i+1}} = \orient{xy}$ and
$\orient{\omega_{j-1}\omega_{j}} = \orient{yx}$. Suppose $\omega$ has
length $n$. Define $\iota$ by $\iota(\omega) = (\omega_{1}, \dots,
\omega_{i+1}, \omega_{j-2}, \dots, \omega_{i+2}, \omega_{j-1},
\omega_{j}, \dots, \omega_{n})$.

\begin{lemma}
  \label{lem:Non-Acceptable}
  $\iota$ is an involution on walks that are not acceptable.
\end{lemma}
\begin{proof}
  It suffices to show that $\iota(\omega)$ is not acceptable if
  $\omega$ is not acceptable. Let $j$ be the index of $\omega$ such
  that $\omega_{j-1}\omega_{j}$ is the first time an edge has been used with
  both orientations. The subwalk $(\omega_{1}, \dots, \omega_{j-1})$ is
  thus edge-simple, and hence $(\iota(\omega)_{1}, \dots,
  \iota(\omega)_{j-1})$ is edge-simple. Since $(\iota(\omega)_{1}, \dots,
  \iota(\omega)_{j})$ is not edge-simple, it follows that $\iota(\omega)$
  is not acceptable.
\end{proof}

A natural scheme to extend $\iota$ to an involution $\hat\iota$ on the
set of walks that are not totally non-backtracking is to perform loop
erasure until the walk is not acceptable, apply $\iota$, then perform
loop addition to put the removed loops back in. This is morally
correct, but since loop addition has only been defined when beginning
with an edge-simple walk, the scheme is not well defined. The precise
construction of $\hat\iota$ follows.

\begin{definition}
  \label{def:K-Acceptable}
  A non-backtracking walk $\omega$ is \emph{$k$-acceptable} if
  $\LE^{k-1}(\omega)$ is acceptable, but $\LE^{k}(\omega)$ is not
  acceptable.
\end{definition}

\begin{lemma}
  \label{lem:Non-Acceptable-Bijection}
  The set of $k$-acceptable non-backtracking walks is in bijection
  with pairs $(\omega^{\star},H)$ where
  \begin{enumerate}
  \item $\omega^{\star}$ is not an acceptable walk.
  \item Suppose $(\omega_{1}^{\star}, \dots, \omega_{m}^{\star})$ is
    the largest edge-simple subwalk of $\omega^{\star}$ beginning at
    $\omega_{1}^{\star}$. Then $H$ is a heap of $k$ loops whose
    maximal elements' labels contain an edge in $(\omega_{1}^{\star},
    \dots, \omega^{\star}_{m})$.
  \end{enumerate}
\end{lemma}
\begin{proof}
  Loop erasure applied $k$ times to a $k$-acceptable walk evidently
  produces such a pair. To recreate the $k$-acceptable walk apply loop
  addition to $((\omega^{\star}_{1}, \dots, \omega^{\star}_{m}),H)$,
  and then append $(\omega^{\star}_{m+1}, \dots, \omega^{\star}_{n})$
  to the resulting walk. As loop erasure and addition do not alter a
  walk after the occurrence of the $k$ loops that are inserted or
  removed, this is a bijection.
\end{proof}

If $\gamma$ is a $k$-acceptable walk, let $(\omega^{\star},H)$ be the
pair corresponding to $\gamma$ as given
by~\Cref{lem:Non-Acceptable-Bijection}.  Define $\hat\iota(\gamma)$ to
be the walk corresponding to $(\iota(\omega^{\star}),H)$. By
construction this walk is $k$-acceptable, as removing $k$ loops
results in the walk $\iota(\omega^{\star})$, which is not acceptable
as $\iota$ is an involution on walks that are not acceptable. Finally,
$\hat\iota$ can be extended to an involution on all non-backtracking
walks by defining it to be the identity map on the set of totally
non-backtracking walks.

\subsection{Main Theorem and Corollaries for the Ising Model}
\label{sec:Applications}

\subsubsection{Proof of Main Theorem}
\label{sec:proof-main-theorem}

To prove the main theorem requires a lemma showing that the weight on
loops is compatible with loop erasure and loop addition, and hence
with the bijection between pyramids of loops and totally
non-backtracking walks. Define weights $w_{\alpha} \colon \tfWalks\to
\C$ for $\alpha\in H_{1}(\surface, \Z_{2})$ by
\begin{equation}
  \label{eq:Weights}
  w_{\alpha}(\gamma) \equiv (-1)^{\tau(\gamma)+\ab{\gamma,\alpha}}
  \prod_{xy\in \gamma} K_{xy}, 
\end{equation}
where $\tau(\gamma)$ is the turning number of $\gamma$ on $\surface$,
and $xy\in\gamma$ is shorthand for the edges contained in the walk
$\gamma$.

\begin{lemma}
  \label{lem:TNB-Weight}
  Let $\gamma_{1}$ be a tail free non-backtracking walk, and let
  $\gamma_{2}$ be a loop sharing an edge with $\gamma_{1}$. Let
  $\gamma = \gamma_{1}\loopadd\gamma_{2}$. Then
  \begin{equation}
    w_{\alpha}(\gamma) = w_{\alpha} (\gamma_{1}) w_{\alpha}(\gamma_{2}).
  \end{equation}
\end{lemma}
\begin{proof}
  Consider $\gamma_{1}$, $\gamma_{2}$ as curves, and perturb
  $\gamma_{2}$ slightly so that it is in general position with respect
  to $\gamma_{1}$. The loop addition operation is the inverse of the surgery
  operation defined in Section~\ref{sec:Torus-Intersections}, so the
  claim follows from Lemma~\ref{lem:Surgery-Points}.
\end{proof}

The last step before proving the main theorem is to establish that the involution
$\hat\iota$ of~\Cref{sec:Involution} is odd with respect to the
weights $w_{\alpha}$ on walks that are not totally non-backtracking.

\begin{lemma}
  \label{lem:Not-TNB}
  Let $\alpha\in H_{1}(\surface,\Z_{2})$. The sum of
  $w_{\alpha}(\gamma)$ over all tail free non-backtracking walks
  that are not totally non-backtracking is zero. That is,
  \begin{equation}
    \sum_{\gamma\in\tfWalks\setminus\tnbWalks}w_{\alpha}(\gamma) = 0.
  \end{equation}
\end{lemma}
\begin{proof}
  As the weight $w_{\alpha}$ is compatible with loop erasure and
  addition by~\Cref{lem:TNB-Weight}, it suffices to show that
  $w_{\alpha}(\gamma)=-w_{\alpha}(\iota(\gamma))$ when $\gamma$ is not
  an acceptable walk. By performing surgery, using the additivity of
  the turning number, and the orientation independence (mod 2) of the
  turning number, this follows from a local calculation in the plane,
  see Figure~\ref{fig:Not-TNB-Cancellation}.
\end{proof}

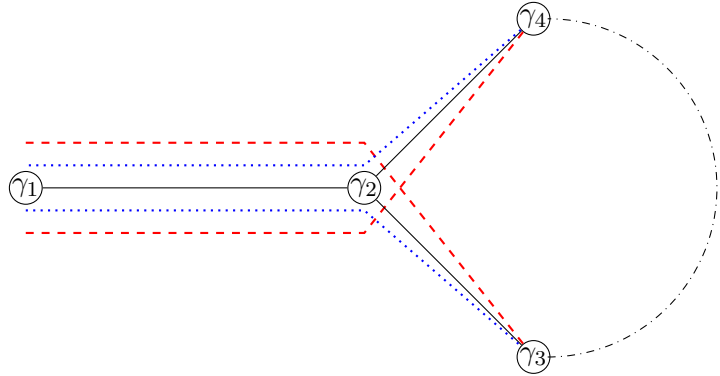
\begin{figure}[h]
  \centering
  \beginpgfgraphicnamed{Figure11}
  \begin{tikzpicture}[scale=1.5]
    \node[ivertex]  (v1) at (0,0) {$\gamma_{1}$};
    \node[ivertex]  (v2) at (3,0) {$\gamma_{2}$};
    \node[ivertex]  (v3) at (4.5,-1.5) {$\gamma_{3}$};
    \node[ivertex]  (v4) at (4.5,1.5) {$\gamma_{4}$};
    \coordinate  (v1pp) at (0,.4) {};
    \coordinate (v1mm) at (0,-.4) {};
    \coordinate  (v2pp) at (3,.4) {};
    \coordinate (v2mm) at (3,-.4) {};
    \coordinate  (v1p) at (0,.2) {};
    \coordinate (v1m) at (0,-.2) {};
    \coordinate  (v2p) at (3,.2) {};
    \coordinate (v2m) at (3,-.2) {};
    \draw[red, dashed, thick] (v1mm) -- (v2mm) -- (v4);
    \draw[red, dashed, thick] (v3) -- (v2pp) -- (v1pp);
    \draw[] (v1) -- (v2) -- (v3);
    \draw[] (v2) -- (v4);
    \draw[dashdotted] (v3) ++(.125,0) arc (-90:90:1.5);
    \draw[blue, dotted, thick] (v1m) -- (v2m) -- (v3);
    \draw[blue, dotted, thick] (v4) -- (v2p) -- (v1p);
  \end{tikzpicture}
  \endpgfgraphicnamed
  \caption{The effect of the involution $\iota$ on a walk that is not
    acceptable.  Here $\gamma_{2}=\gamma_{k-1}$ and
    $\gamma_{1}=\gamma_{k}$. The dashed-and-dotted arc represents the
    portion of the walk not drawn. A perturbation of $\gamma$ is drawn
    as a dotted blue curve, while a perturbation of $\iota(\gamma)$ is
    drawn as a dashed red curve. Both the blue and red paths begin on
    their lower segments.}
  \label{fig:Not-TNB-Cancellation}
\end{figure}

\begin{lemma}
  \label{lem:Main-Even}
  Let $\alpha\in H_{1}(\surface, \Z_{2})$. Then for a graph $G$
  properly embedded on $\surface$ 
  \begin{equation}
    \label{eq:Pointing-Pyramids}
    \log  \sum_{T\in\cc T(\bCSWalks,\cap_{e})}(-1)^{\abs{T}}w_{\alpha}(T) =
    -\sum_{\gamma\in\tfWalks}\frac{w_{\alpha}(\gamma)}{2\abs{\gamma}}.
  \end{equation}
  where $\cc T(\bCSWalks,\cap_{e})$ denotes the set of trivial heaps
  of loops with concurrency relation edge intersection.
\end{lemma}
\begin{proof}
  Let $Z_{G}=\sum_{T\in\cc T(\bCSWalks(G),\cap_{e})}
  (-1)^{\abs{T}}w_{\alpha}(T)$. Let $K$ stand for the collection of
  variables $\{K_{xy}\}_{xy\in E(G)}$ used in the definition of the
  weight on loops, and let $tK =
  \{tK_{xy}\}$. The fundamental theorem of calculus implies
    \begin{equation}
      \label{eq:t-Derivative}
      \log Z_{G}(K) = \int_{0}^{1} \frac{d}{dt} \log Z_{G}(tK)\,dt,
    \end{equation}
    where $Z_{G}(K)$ indicates the weight on loops uses the variables
    $K$.  If a trivial heap of loops has $n$ edges contained in its
    labels, the factor of $n$ resulting from the derivative with
    respect to $t$ in~\eqref{eq:t-Derivative} can be replaced by
    summing over the edges contained in the trivial heap:
  \begin{equation}
      \label{eq:Heap-Expression}
      \log Z_{G}(K) = \int_{0}^{1} -\frac{1}{t Z_{G}(tK)} \sum_{xy\in E(G)} 
    \sum_{C\colon xy\in E(C)} w_{\alpha}(C)
    \sum_{T\in\mathcal{T}_{C}} (-1)^{\abs{T}}w_{\alpha}(T)\,dt.
    \end{equation}
    The second sum is over loops $C$ containing $xy$, and
    $\mathcal{T}_{C}$ is the set of trivial heaps whose elements do
    not intersect the loop $C$. The weight $w_{\alpha}$ on the
    right-hand side of~\eqref{eq:Heap-Expression} uses the variables
    $tK$. The remainder of the proof consists of manipulating the
    right-hand side of~\eqref{eq:Heap-Expression}, and to simplify
    expressions we will not make the $t$ dependence explicit.

  Equation~\eqref{eq:Heap-Expression} is powerful, as the
  heap theorem immediately yields a simple expression for the ratio
  of partition functions in~\eqref{eq:Heap-Expression}. More precisely, 
  Theorem~\ref{thm:Heaps-Theorem} implies
  \begin{equation}
    \label{eq:OUO-Pre}
    Z_{G}^{-1} \sum_{T\in \mathcal{T}_{C}} (-1)^{\abs{T}}w_{\alpha}(T)
    = \sum_{H\in\cc H_{C}} w_{\alpha}(H),
  \end{equation}
  where $\cc H_{C}$ is the set of heaps of loops whose maximal
  elements intersect $C$. Hence~\eqref{eq:Heap-Expression} simplifies
  to
  \begin{equation}
    \label{eq:OUO-Pyramids.1}
    -\frac{1}{tZ_{G}}\sum_{xy\in E(G)} \sum_{C\colon xy\in E(C)} w_{\alpha}(C)
    \sum_{T\in\mathcal{T}_{C}} (-1)^{\abs{T}}w_{\alpha}(T) = -
    \frac{1}{t} \sum_{xy\in E(G)}
    \sum_{C\colon xy\in E(C)}\sum_{H\in\cc H_{C}}
    w_{\alpha}(C)w_{\alpha}(H),
  \end{equation}
  The double sum in equation~\eqref{eq:OUO-Pyramids.1} equals $1/2$
  the sum over pairs $(\orient{xy},P)$ and $(\orient{yx},P)$, where
  $P$ is a pyramid whose maximal elements' label is $C$.
  \Cref{cor:Walks-Pyramids} allows us to rewrite this as a sum over
  totally non-backtracking walks with initial edges $\orient{xy}$ or
  $\orient{yx}$, and~\Cref{lem:Not-TNB} allows us to expand the range
  of the sum to all tail free non-backtracking walks with initial edge
  $\orient{xy}$ or $\orient{yx}$. If a given walk contains
  $\abs{\gamma}$ steps then the weight of the walk is
  $t^{\abs{\gamma}}$ times the weight using the $K$
  variables. Integrating $\frac{1}{t} t^{\abs{\gamma}}$ from $0$ to
  $1$ results in a multiplicative factor of $1/\abs{\gamma}$. Putting
  this together gives
  \begin{equation}
   \log Z_{G}= -\sum_{\gamma\in\tfWalks}
   \frac{w_{\alpha}(\gamma)}{2\abs{\gamma}}. \qedhere
  \end{equation}
\end{proof}

Recall that for $\alpha\in H_{1}(\surface,\Z_{2})$ the value
$c(\alpha)$ is defined by
\begin{equation}
  (-1)^{c(\alpha)}=\prod_{j=1}^{g}(-1)^{1+\ab{\alpha,e^{j}_{1}}+\ab{\alpha,e^{j}_{2}}
  + \ab{\alpha,e^{j}_{1}}\ab{\alpha,e^{j}_{2}}}.
\end{equation}
The central theorem of the paper,
Theorem~\ref{thm:Introduction-Main-Even}, can now be proven.

\ThmOne*
\begin{proof}
  Apply Lemma~\ref{lem:Main-Even} to
  Lemma~\ref{lem:Matchings-Decompositions}, using
  Corollary~\ref{cor:Decomposition-Weight-Prime}.
\end{proof}

\subsubsection{Some Consequences for the Ising Model}
\label{sec:cons-ising-model}

Throughout this section assume $G$ is a finite graph that is properly
embedded in $\surface$, a genus $g$ surface. The Ising model on $G$
will refer to the Ising model with couplings $\{L_{xy}\}_{xy\in
  E(G)}$, and $K_{xy}\equiv \tanh L_{xy}$. Proposition~\ref{prop:HTE}
combined with Theorem~\ref{thm:Introduction-Main-Even} immediately
gives a formula for the Ising model partition function.

\begin{theorem}
  \label{thm:Ising-Z-Surfaces}
  The partition function $Z(G)$ for the Ising model on a finite graph
  $G$ can be written
  \begin{equation*}
    Z(G) = \frac{1}{2^{g}}\sum_{\alpha\in
      H_{1}(\surface,\Z_{2})}(-1)^{c(\alpha)}\exp\ob{ -
    \sum_{\gamma\in\tfWalks} \frac{w_{\alpha}(\gamma)} {2\abs{\gamma}}}.
  \end{equation*}
  where $c(\alpha)$ is given by Equation~\eqref{eq:Constant-c} and
  $w_{\alpha}$ is defined by Equation~\eqref{eq:Weights}.
\end{theorem}

\begin{remark}
  \label{rem:Ising-Solution}
  By approximating $\Z^{2}$ by finite subsets and applying
  the above formula the Onsager solution to the Ising
  model can be recovered. See~\cite{Sherman1960} for details.
\end{remark}

As remarked in the introduction, the fact that
Theorem~\ref{thm:Ising-Z-Surfaces} applies to non-planar graphs allows
for the derivation of formulas for spin correlations. Let
$\ab{\sigma_{a}\sigma_{b}}_{H}$ denote the expectation of the product
of the spins at $a$ and $b$ in the Ising model on the graph
$H$. Define $G_{ab}$ to be the graph with vertex set $V(G)$ and edge
set $E(G)\cup\{ab\}$. A calculation using the definition of the Ising
model shows
\begin{equation}
  \label{eq:Fundamental-Correlation-Identity}
  \ab{\sigma_{a}\sigma_{b}}_{G} =  \frac{\partial}{\partial
    K_{ab}}\Big|_{K_{ab}=0} \log Z(G_{ab}).
\end{equation}

Using Equation~\eqref{eq:Fundamental-Correlation-Identity} along with
Theorem~\ref{thm:Ising-Z-Surfaces} gives a
formula for spin correlations, but the weight on walks is the weight
inherited from a surface on which the graph $G_{ab}$ embeds. The
expressions for correlations can be simplified using some corollaries
of Theorem~\ref{thm:General-Intersection-Theorem}. Some additional
notation will be needed.

Fix $a,b\in V(G)$. Fix dual vertices $a^{\star},b^{\star}\in
V(G^{\star})$ such that the faces corresponding to $a^{\star}$ and
$b^{\star}$ contain $a$ and $b$. Let $\surface^{\prime}_{a,b}$ denote
the surface that results from removing small disks touching
$a^{\star}$ and $b^{\star}$, and attaching a cylinder to the
boundaries of the removed disks. On $\surface^{\prime}_{a,b}$, let
$e^{g+1}_{1},e^{g+1}_{2}$ be closed curves such that the basis of
$H_{1}(\surface,\Z_{2})$ along with these additional curves forms a
basis of $H_{1}(\surface^{\prime}_{a,b},\Z_{2})$, with $e^{g+1}_{1}$
entirely contained in the added cylinder. Let $\eta_{1}$ denote the
portion of $e^{g+1}_{2}$ in $\surface$, and $\eta_{2}$ the portion in
$\surface^{\prime}_{a,b}\setminus\surface$. Choose $e^{g+1}_{2}$ such
that $\eta_{1}$ is contained in the dual graph $G^{\star}$ outside a
small neighbourhood of $a^{\star}$ and $b^{\star}$, and such that the
edge $ab$ does not intersect $\eta_{2}$. See
Figure~\ref{fig:Handle-Surface}.  Recall that $\gamma\eta$ denotes the
concatenation of the curves $\gamma$ and $\eta$.

\begin{figure}
  \centering
  \beginpgfgraphicnamed{Figure12}
  \begin{tikzpicture}
    \begin{scope}
    \clip (-3.1,-2.5) circle (2);
    \draw[black!50] (-4,-4) grid (-1,-.5);
   \end{scope} 
   \begin{scope}
     \clip (5.1,0.1) circle (2);
    \draw[black!50] (3,-2) grid (6,1);
    \end{scope}
    \node[svertex] (a) at (-3,-3) {};
    \node[svertex] (b) at (5,-1) {};
    \node[svertex] (ap) at (-2.51,-3.5) {};
    \node[svertex] (bp) at (4.5,-.5) {};
    \node[black,left] at (a) {$a$};
    \node[black,right] at (b) {$b$};
    \draw[white,very thick,fill] (-2.14,-3.65) .. controls (0,1)
     and (2,2) .. (4.14,-.65);
   \draw[white,very thick,fill] (-2.56,-3.65) ..  controls (0,1.4)
    and (2,2.3) .. (4.56,-.65);
    \draw[black] (-2.14,-3.65) .. controls (0,1)
     and (2,2) .. (4.14,-.65);
   \draw[black] (-2.56,-3.65) ..  controls (0,1.4)
    and (2,2.3) .. (4.56,-.65);
    \draw[dashed] (-2.35,-3.5) .. controls (0,1.2)
     and (2,2.15) .. (4.35,-.5);
    \node[black] at (ap) {$a^{\star}\,\,\,\,\,\,\,\,$};
    \node[black,right] at (bp) {$b^{\star}$};
    \draw (a) ++ (.44,-.62) arc (-180:0:.21);
    \draw[dashed] (a) ++ (.44,-.65) arc (180:0:.21);
    \draw (b) ++ (-.44,.38) arc (0:-180:.21);
    \draw[dashed] (b) ++ (-.44,.35) arc (0:180:.21);
    \draw[dashed] (-2.35,-3.5) -- (-2.35,.5) -- (.25,.5);
    \draw[dashed] (1,.5) -- (2.65,.5);
    \draw[dashed] (3.5,.5) -- (4.35,.5) -- (4.35,-.5);
    \node[above] at (-1.5,.5) {$\eta_{1}$};
  \end{tikzpicture}
  \endpgfgraphicnamed
  \caption{A schematic drawing of a portion of the surface
    $\surface^{\prime}_{a,b}$. The dashed line running along the
    handle is $\eta_{2}$. For visual clarity the edge $ab$ has not
    been drawn. The basis curve $e^{g+1}_{2}$ is comprised of
    $\eta_{1}$ and $\eta_{2}$.  The basis curve $e^{g+1}_{1}$, not
    pictured, is contained entirely on the handle, and cuts the handle
    in two.}
  \label{fig:Handle-Surface}
\end{figure}
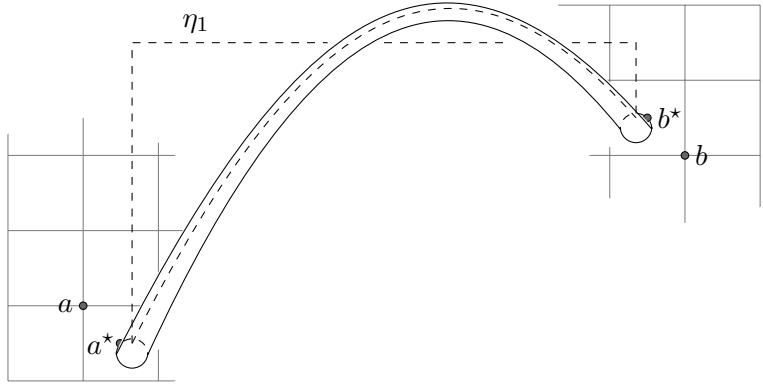

\begin{corollary}
  \label{cor:Turning-Comparison}
  Let $\gamma$ be a closed regular curve contained in
  $\surface\subset\surface^{\prime}_{a,b}$. Then
  \begin{equation*}
    \tau_{\surface^{\prime}_{a,b}}(\gamma) = \tau_{\surface}(\gamma)
    + \ab{\gamma,e^{g+1}_{2}} \mod 2.
  \end{equation*}
\end{corollary}
\begin{proof}
  Apply Theorem~\ref{thm:General-Intersection-Theorem} to the curve
  $\gamma$ twice, once considered as a curve on $\surface$, once
  considered as a curve on $\surface^{\prime}_{a,b}$. When considered
  as a curve on $\surface^{\prime}_{a,b}$, $\gamma$ can intersect only
  $e^{g+1}_{2}$, as $e^{g+1}_{1}$ is disjoint from $\surface$.
\end{proof}

\begin{corollary}
  \label{cor:Turning-Comparison-Handle}
  Let $\gamma$ be a regular curve from $a^{\star}$ to $b^{\star}$ in
  $\surface$. Then
  \begin{equation*}
    \tau_{\surface^{\prime}_{a,b}}(\gamma\eta_{2})=
    \tau_{\surface}(\gamma\eta_{1}) +
    \abs{\Delta_{2}(\{\gamma,\eta_{1}\})} + 
    1 \mod 2.
  \end{equation*}
\end{corollary}
\begin{proof}
  Let $\gamma_{i}=\gamma\eta_{i}$. By the definitions of the curves given,
  $\ab{\gamma_{2},e^{g+1}_{1}}=1$,
  $\ab{\gamma_{1},e^{g+1}_{1}}=0$, and
  $\gamma_{2}=\gamma_{1}+e^{g+1}_{2}$ in $\Z_{2}$ homology.  The curves
  $\gamma_{1}$ and $\gamma_{2}$ have the same set of double
  points aside from when $\gamma$ intersects $\eta_{1}$ so
  Theorem~\ref{thm:General-Intersection-Theorem} implies
  \begin{align*}
    \tau_{\surface^{\prime}_{a,b}}(\gamma_{2}) &= N_{+}(\gamma_{2}) -
    N_{-}(\gamma_{2})-\sum_{j=1}^{g+1} \ob{\ab{\gamma_{2},e^{j}_{1}}
      \ab{\gamma_{2},e^{j}_{2}} + \ab{\gamma_{2},e^{j}_{1}} +
      \ab{\gamma_{2},e^{j}_{2}}} -1 \mod 2
    \\
    &= N_{+}(\gamma_{1}) - N_{-}(\gamma_{1}) +
    \abs{\Delta_{2}(\{\gamma,\eta_{1}\})} -
    \sum_{j=1}^{g+1}\ob{\ab{\gamma_{1},e^{j}_{1}}
      \ab{\gamma_{1},e^{j}_{2}} + \ab{\gamma_{1},e^{j}_{1}} +
      \ab{\gamma_{1},e^{j}_{2}}} -
    \ab{\gamma_{1},e^{g+1}_{2}} \mod 2 \\
    &= \tau_{\surface^{\prime}_{a,b}}(\gamma_{1}) +
    \abs{\Delta_{2}(\{\gamma,\eta_{1}\})} - \ab{\gamma_{1},e^{g+1}_{2}} + 1 \mod 2\\
    &= \tau_{\surface}(\gamma_{1}) +
    \abs{\Delta_{2}(\{\gamma,\eta_{1}\})} + 1 \mod 2,
  \end{align*}
  where the second equality is an intersection form calculation, the
  third is another application of
  Theorem~\ref{thm:General-Intersection-Theorem}, and the final
  equality is by Corollary~\ref{cor:Turning-Comparison}.
\end{proof}

Define couplings $\tilde K_{xy}$ by
\begin{equation*}
  \tilde K_{xy} =
    \begin{cases}
      \phantom{-}K_{xy} & (xy)^{\star}\notin \eta_{1} \\
      -K_{xy} & (xy)^{\star}\in\eta_{1}.
    \end{cases}. 
\end{equation*}

\begin{proposition}
  \label{prop:Intersections-Dislocation}
  For all walks $\gamma$
  \begin{equation*}
    (-1)^{\abs{\Delta_{2}(\{\gamma,\eta_{1}\})}} =
    \mathrm{sign}\ob{ \frac{\prod_{xy\in \gamma}\tilde
      K_{xy}}{\prod_{xy\in\gamma}K_{xy}}}.
  \end{equation*}
\end{proposition}
\begin{proof}
  By the definition of $\eta_{1}$ an edge $xy$ of $\gamma$ crosses $\eta_{1}$
  if and only if the edge $xy$ is dual to an edge in $\eta_{1}$.
\end{proof}

If $\gamma$ is a non-backtracking walk from $a$ to $b$, extend
$\gamma$ to begin at $a^{\star}$ and end at $b^{\star}$. With this
extension the walk $\gamma$ can be concatenated with
$\eta_{2}$, and the turning number of $\gamma$ is defined to be the
turning number of the closed curve $\gamma\eta_{2}$. Let $\tilde
w_{\alpha}(\gamma)$ be the weight of a walk when the couplings $K_{xy}$
are replaced by $\tilde K_{xy}$.
\ThmTwo*

\begin{remark}
  \label{rem:Correlation-Remark}
  It is worth stressing that
  Equation~\eqref{eq:Introduction-Correlation-Formula-HT} requires the
  walks to be completed to closed walks. While the geometric
  assumptions on the intersections of the edge $ab$ and the generator
  $e^{g+1}_{2}$ are not essential, altering these assumptions may
  alter the correlation formula.
\end{remark}

\begin{proof}[Proof of Theorem~\ref{thm:Introduction-Correlation-Formula-HT}]
  We use Theorem~\ref{thm:Ising-Z-Surfaces} and
  Equation~\eqref{eq:Fundamental-Correlation-Identity}.  Note that
  $Z(G_{ab})$ at $K_{ab}=0$ is $Z(G)$. To calculate the derivative in
  Equation~\eqref{eq:Fundamental-Correlation-Identity}, note that if a
  loop $\gamma$ contains the edge $ab$ exactly once, there are
  $2\abs{\gamma}$ walks that are equivalent to $\gamma$. As only walks
  which contain the edge $ab$ exactly once survive being
  differentiated and having $K_{ab}$ set to zero
  \begin{equation}
    \label{eq:Correlation-High-1}
    \ab{\sigma_{a}\sigma_{b}}_{G} = -\frac{1}{2^{g+1}Z(G)}
    \sum_{\alpha\in H_{1}(\surface^{\prime}_{a,b},\Z_{2})}
    \sum_{\gamma\in\nbWalks(G,a,b)} (-1)^{c(\alpha)} w^{\prime}_{\alpha}(\gamma) \exp\ob{-
    \sum_{\gamma\in\tfWalks(G)} \frac{w^{\prime}_{\alpha}(\gamma)}{2\abs{\gamma}}},
  \end{equation}
  where $K_{ab}=1$, and walks from $a$ to $b$ are taken to be
  closed by adding the edge $ba$ at the end. The notation
  $w^{\prime}_{\alpha}$ indicates the weights are computed using
  the turning number on the surface $\surface^{\prime}_{a,b}$.

  For $\alpha\in H_{1}(\surface^{\prime},\Z_{2})$, let
  $\alpha=\beta + c_{1}e^{g+1}_{1} + c_{2}e^{g+1}_{2}$, where $\beta\in
  H_{1}(\surface,\Z_{2})$. First consider the sum in the
  exponential. Each contributing walk is contained in
  $\surface\subset \surface^{\prime}$; applying
  Corollary~\ref{cor:Turning-Comparison} gives
  \begin{align}
    \label{eq:Turning-Exp-1}
    \tau_{\surface^{\prime}_{a,b}}(\gamma) + \ab{\gamma,\alpha} &=
    \tau_{\surface}(\gamma) + c_{1}\ab{\gamma,e^{g+1}_{1}} +
    (1+c_{2})\ab{\gamma,e^{g+1}_{2}}
    \\
    \label{eq:Turning-Exp-2}
    &= \tau_{\surface}(\gamma) + (1+c_{2})\ab{\gamma,e^{g+1}_{2}},
  \end{align}
  as $\ab{\gamma,e^{g+1}_{1}} = 0$ for a walk contained in $\surface$.
  For the sum over walks in $\nbWalks(G,a,b)$ outside the exponential,
  Corollary~\ref{cor:Turning-Comparison-Handle} and the fact that
  $\gamma\eta_{2} = \gamma\eta_{1} + e^{g+1}_{2}$ in homology
  yields
  \begin{align}
    \label{eq:Turning-Exp-3.1}
    \tau_{\surface^{\prime}_{a,b}}(\gamma\eta_{2}) &=
    \tau_{\surface}(\gamma\eta_{1}) +
    \abs{\Delta_{2}(\{\gamma,\eta_{1}\})} + 1 \mod 2 \\
    \label{eq:Turning-Exp-3.2}
    \ab{\gamma\eta_{2},\alpha} &= \ab{\gamma\eta_{1},\beta} +
    c_{1}\ab{\gamma\eta_{1},e^{g+1}_{1}} +
    c_{2}\ab{\gamma\eta_{1},e^{g+1}_{2}} + c_{1} \mod 2 \\
    \label{eq:Turning-Exp-3.3}
    c(\alpha) &= c(\beta) + 1 + c_{1} + c_{2} + c_{1}c_{2},
  \end{align}
  where the second equality has used $\ab{e^{g+1}_{2},\alpha} =
  c_{1}$. Summing~\eqref{eq:Turning-Exp-3.1},
  \eqref{eq:Turning-Exp-3.2} and~\eqref{eq:Turning-Exp-3.3} gives
  \begin{equation}
    \label{eq:Turning-Exp-3}
    c(\alpha) + \tau_{\surface^{\prime}_{a,b}}(\gamma\eta_{2}) +
    \ab{\gamma\eta_{2},\alpha} =c(\beta)+ \tau_{\surface}(\gamma\eta_{1}) +
    \ab{\gamma\eta_{1},\beta} + \abs{\Delta_{2}(\{\gamma,\eta_{1}\})}
    + c_{2}\ob{1+c_{1}+ \ab{\gamma\eta_{1},e^{g+1}_{2}}} \mod 2
  \end{equation}
  as $\ab{\gamma\eta_{1},e^{g+1}_{1}}=0$, since $\gamma$ concatenated
  with $\eta_{1}$ is contained in $\surface$. The weight in the
  exponent given by~\eqref{eq:Turning-Exp-2}
  is independent of $c_{1}$, while the weight on walks in
  $\nbWalks(G,a,b)$ is odd in $c_{1}$ if $c_{2}=1$. Thus only $\alpha$
  with $c_{2}=0$ contribute. Therefore
  \begin{align*}
    \label{eq:Surface-Weight-Sum}
    \sum_{\alpha\in H_{1}(\surface^{\prime},\Z_{2})} (-1)^{c(\alpha)}
    w_{\alpha}^{\prime}(\gamma) &\exp\ob{- \sum_{\gamma\in\tfWalks(G)}
      \frac{w^{\prime}_{\alpha}(\gamma)}{2\abs{\gamma}}}\\ &=
    \sum_{\beta\in H_{1}(\surface,\Z_{2})}
    \sum_{\{c_{1},c_{2}\}\in\{0,1\}^{2}} w_{\alpha}^{\prime}(\gamma) \exp\ob{-
      \sum_{\gamma\in\tfWalks(G)}
      \frac{w^{\prime}_{\alpha}(\gamma)}{2\abs{\gamma}}} \\
    &= 2\sum_{\beta\in H_{1}(\surface,\Z_{2})}
    (-1)^{c(\beta)}w_{\beta}(\gamma) (-1)^{
      \abs{\Delta_{2}(\{\gamma,\eta_{1}\})}} \exp\ob{-
      \sum_{\gamma\in\tfWalks(G)}(-1)^{\ab{\gamma,e^{g+1}_{2}}}
      \frac{w_{\beta}(\gamma)}{2\abs{\gamma}}}\\
    &= 2\sum_{\beta\in H_{1}(\surface,\Z_{2}) } (-1)^{c(\beta)}\tilde
    w_{\beta}(\gamma) \exp\ob{- \sum_{\gamma\in\tfWalks(G)}
      \frac{\tilde w_{\beta}(\gamma)}{2\abs{\gamma}}},
  \end{align*}
  where the use of $\tilde w_{\beta}$ follows from
  Proposition~\ref{prop:Intersections-Dislocation} and the fact that a
  closed walk in $\surface$ can only intersect $e^{g+1}_{2}$ by
  intersecting $\eta_{1}$.
\end{proof}

Corollary~\ref{cor:High-Temperature-Planar-Correlation} for the
spin-spin correlations of a planar Ising model follows by
specializing the previous result to the planar case.

\begin{remark}
  \label{rem:Other-Correlations}
  While Theorem~\ref{thm:Introduction-Correlation-Formula-HT} does not
  apply to correlations of neighbouring spins, formulas for these
  correlations can be derived in a similar but simpler manner.
\end{remark}

\subsubsection{Planar Ising Duality}
\label{sec:Planar-Duality}

Let $G$ be a properly embedded planar graph. As in the previous section let
$a,b\in V(G)$ be fixed vertices, and $a^{\star},b^{\star}\in
V(G^{\star})$ be dual vertices such that $a$ and $b$ are contained in the
faces corresponding to $a^{\star}$ and $b^{\star}$. Let $\eta$ be a
path from $a^{\star}$ to $b^{\star}$ contained in the dual graph.

The analysis of Ising models so far has relied on
Proposition~\ref{prop:HTE} to rewrite the partition function of an
Ising model as a generating function of even subgraphs. For the Ising
model with plus boundary conditions on planar graphs there is also a
representation in terms of even subgraphs, the \emph{low-temperature
  expansion}~\cite{Baxter1982}:

\begin{lemma}[name = Low Temperature Expansion]
  \label{lem:Low-Temperature-Expansion}
  Let $\{L_{(xy)^{\star}}\}$ be an Ising model on $G^{\star}$, the
  dual of $G$. Up to an irrelevant multiplicative factor the partition function of
  the Ising model on $G^{\star}$ with plus boundary conditions is given by
  \begin{equation*}
    Z(G) = \sum_{H\in \cc E(G)}w^{\star}(H), \qquad w^{\star}(H) = \prod_{xy\in
      E(H)}\exp(-2L_{(xy)^{\star}}).
  \end{equation*}
\end{lemma}

Let $K^{\star}_{xy} = \exp(-2L_{(xy)^{\star}})$, and define weights
$\tilde K^{\star}_{xy}$ on the edges $xy$ of $G$ as in
Section~\ref{sec:cons-ising-model}, i.e., if $xy$ is dual to an edge
in $\eta$, $\tilde K^{\star}_{xy}=-K^{\star}_{xy}$, and $\tilde
K^{\star}_{xy}=K^{\star}_{xy}$ otherwise.

\begin{corollary}
  \label{cor:Low-Temperature-Correlation}
  The spin-spin correlation $\ab{\sigma_{a^{\star}}
    \sigma_{b^{\star}}}$ is given by
  \begin{align*}
    \ab{\sigma_{a^{\star}}\sigma_{b^{\star}}} &=
    \frac{1}{Z(G)}\sum_{H\in\cc E(G)}\prod_{xy\in E(H)} \tilde K^{\star}_{xy} \\
    &= \exp\ob{-\sum_{\gamma\in\nbWalks(G)}\frac{\tilde
      w^{\star}(\gamma)-w^{\star}(\gamma)}{2\abs{\gamma}}}.
  \end{align*}
\end{corollary}
\begin{proof}
  The first line follows from a similar analysis as in the proof of
  Lemma~\ref{lem:Low-Temperature-Expansion}, noting that the sign of
  $\sigma_{a^{\star}}\sigma_{b^{\star}}$ can be determined by counting
  the number of times a path from $a^{\star}$ to $b^{\star}$ is
  intersected by the even subgraph $H$. The second line then follows by applying
  Theorem~\ref{thm:Ising-Z-Surfaces}.
\end{proof}

Combined with Corollary~\ref{cor:High-Temperature-Planar-Correlation},
the previous result immediately yields a duality between
correlations on $G$ and $G^{\star}$.
\begin{theorem}
  \label{thm:Planar-Duality}
  Let $\ab{\sigma_{a}\sigma_{b}}_{G}$ denote the correlation of
  $\sigma_{a}$ and $\sigma_{b}$ for the Ising model $\{L_{xy}\}$ on a
  planar graph $G$, and let
  $\ab{\sigma_{a^{\star}}\sigma_{b^{\star}}}_{G^{\star}}$ denote the
  correlation of $\sigma_{a^{\star}}$ and $\sigma_{b^{\star}}$ for the Ising model
  $\{L_{(xy)^{\star}}\}$ on $G^{\star}$. Then if
  $L_{(xy)^{\star}}= -\frac{1}{2}\log\tanh L_{xy}$
  \begin{equation}
    \label{eq:Planar-Duality}
    \frac{\ab{\sigma_{a}\sigma_{b}}_{G}}
    {\ab{\sigma_{a^{\star}}\sigma_{b^{\star}}}_{G^{\star}}} = 
    -\sum_{\gamma\in\nbWalks(G,a,b)} \tilde w(\gamma),
  \end{equation}
  where $K_{xy}=\tanh L_{xy} = \exp\ob{-2L_{(xy)^{\star}}} = K^{\star}_{xy}$.
\end{theorem}

Equation~\eqref{eq:Planar-Duality} shows that, up to a modulus one
multiplicative constant, the sum over walks on the right-hand side of
Equation~\eqref{eq:Planar-Duality} is the value of the holomorphic
fermionic observable discussed
in~\cite[Eq.~(2.5)]{ChelkakHonglerIzyurov2012},
yielding~\Cref{cor:Spinor-Identification}.

\section{Acknowledgements}
\label{sec:Ack}

The author would like to thank his supervisor, Professor David
Brydges, both for suggesting this problem and for his invaluable
support, comments and advice. The author is also very grateful to the
anonymous referee whose comments and criticisms led to important
corrections and improvements.

\appendix
\section{Heaps of Pieces}
\label{sec:Heaps}

See either of~\cite{Krattenthaler2000,Viennot1986} for an introduction
to the theory of heaps of pieces. The notation used here is the same,
though some of the terminology is slightly different. To interpret the
following definitions let $\cc B$ be the set of all distinct
loops. Loops in $\cc B$ are concurrent when they have edges in
common. A heap of loops is visualized literally as a heap. In
particular, when two loops are concurrent then one must be above the
other. Heaps may contain the same loop multiple times. Pyramids are
heaps with a unique highest loop.

A \emph{concurrency relation} is a symmetric and reflexive binary
relation. Let $\cc R$ be a concurrency relation on a set $\cc B$. The
set $\cc B$ will be called the set of \emph{piece types}. A
heap of pieces $(H,\preceq,\ell)$ is a triple with $\ell\colon H\to
\cc B$ and $(H,\preceq)$ a poset such that
\begin{enumerate}
\item If $x,y\in H$ and $\ell(x)\cc R\ell(y)$ then either $x\preceq y$
  or $y\preceq x$.
\item The relation $\preceq$ is the transitive closure of the
  relations from the previous condition.
\end{enumerate}

The map $\ell$ is called the \emph{labelling} of the pieces. Given a
collection of piece types $\cc B$ and a concurrency relation $\cc R$ define
\begin{itemize}
\item $\cc H(\cc B,\cc R)$ to be the set of all heaps of pieces,
\item $\cc T(\cc B,\cc R)$ to be the set of \emph{trivial heaps} of pieces,
  i.e., heaps of pieces for which $\ell(x)\!\!\not{\!\!\cc R} \ell(y)$ for any
  $x,y\in H$.
\item $\cc P(\cc B,\cc R)$ to be the set of \emph{pyramids}, i.e., heaps
  of pieces which contain a unique maximal element.
\end{itemize}

There is a natural notion of composition for two heaps of
pieces. If $(H_{i},\preceq_{i},\ell_{i})\in\cc H(\cc B,\cc R)$ then define
\begin{equation}
  (H_{1},\preceq_{1},\ell_{1}) \circ (H_{2},\preceq_{2},\ell_{2}) = (H_{3},\preceq_{3},\ell_{3})
\end{equation}
where
\begin{enumerate}
\item $H_{3} = H_{1}\cup H_{2}$
\item The partial order $\preceq_{3}$ is the transitive closure of
  \begin{itemize}
  \item $v_{1}\preceq_{3}v_{2}$ if $v_{1}\preceq_{1}v_{2}$,
  \item $v_{1}\preceq_{3}v_{2}$ if $v_{1}\preceq_{2}v_{2}$,
  \item $v_{1}\preceq_{3}v_{2}$ if $v_{1}\in H_{1}$, $v_{2}\in
    H_{2}$, and $\ell_{1}(v_{1})\cc R\ell_{2}(v_{2})$.
  \end{itemize}
\end{enumerate}

Intuitively the heap $(H_{2},\preceq_{2},\ell_{2})$ is
placed on top of the heap $(H_{1},\preceq_{1},\ell_{1})$. Note that
this addition is not commutative.

If $\cc M\subset\cc B$ then $\cc H_{\cc M}(\cc B,\cc R)$ is the set of
heaps of pieces with piece types $\cc B$ whose maximal elements have
labels in $\cc M$.

A weight function $w\colon\cc B\to \C$ naturally extends to heaps of
pieces via
\begin{equation}
  w(H) = \prod_{x\in H}w(\ell(x)).
\end{equation}

The result from the theory of heaps of pieces that is needed for this
paper is the following. See~\cite{Krattenthaler2000} for a short and
elegant proof.

\begin{theorem}
  \label{thm:Heaps-Theorem}
  As formal power series
  \begin{equation}
    \sum_{H\in\cc H_{\cc M}(\cc B,\cc R)} w(H) = \frac{ \sum_{T\in \cc
      T(\cc B\setminus \cc M,\cc R)} (-1)^{\abs{T}}w(T)} { \sum_{T\in
      \cc T(\cc B,\cc R)} (-1)^{\abs{T}}w(T)}
  \end{equation}
\end{theorem}

\bibliographystyle{ieeetr}

\begin{thebibliography}{10}

\bibitem{KacWard1952}
M.~Kac and J.~Ward, ``{A Combinatorial Solution of the Two-Dimensional Ising
  Model},'' {\em Physical Review}, vol.~88, pp.~1332--1337, 1952.

\bibitem{Sherman1960}
S.~Sherman, ``{Combinatorial Aspects of the Ising Model for Ferromagnetism. I.
  A Conjecture of Feynman on Paths and Graphs},'' {\em Journal of Mathematical
  Physics}, vol.~1, no.~3, p.~202, 1960.

\bibitem{Vdovichenko1965a}
N.~Vdovichenko, ``{A Calculation of the Partition Function for a Plane Dipole
  Lattice},'' {\em Soviet Physics JETP}, vol.~20, no.~2, pp.~477--479, 1965.

\bibitem{Burgoyne1963}
P.~N. Burgoyne, ``{Remarks on the Combinatorial Approach to the Ising
  Problem},'' {\em Journal of Mathematical Physics}, vol.~4, no.~10, p.~1320,
  1963.

\bibitem{daCostaMaciel2003}
G.~da~Costa and A.~Maciel, ``Combinatorial Formulation of Ising Model
  Revisited,'' {\em Revista Brasileira de Ensino de F{\i}sica}, vol.~25, no.~1,
  p.~49, 2003.

\bibitem{Glasser1970}
M.~L. Glasser, ``Exact Partition Function for the Two-Dimensional Ising
  Model,'' {\em American Journal of Physics}, vol.~38, p.~1033, 1970.

\bibitem{DolbilinZinovevMishchenkoShtankoShtogrin1999}
N.~P. Dolbilin, Y.~M. Zinov{\cprime}ev, A.~S. Mishchenko, M.~A.
  Shtan{\cprime}ko, and M.~I. Shtogrin, ``The {K}ac-{W}ard determinant,'' {\em
  Tr. Mat. Inst. Steklova}, vol.~225, no.~Solitony Geom. Topol. na Perekrest.,
  pp.~177--194, 1999.

\bibitem{Loebl2004}
M.~Loebl, ``{A Discrete non-Pfaffian Approach to the Ising Problem},'' in {\em
  DIMACS 63: Graphs, Morphisms and Statistical Physics}, pp.~145--155, 2004.

\bibitem{Cimasoni2010}
D.~Cimasoni, ``{A Generalized Kac-Ward Formula},'' {\em Journal of Statistical
  Mechanics: Theory and Experiment}, vol.~2010, pp.~7--23, 2010.

\bibitem{KagerLisMeester2013}
W.~Kager, M.~Lis, and R.~Meester, ``The Signed Loop Approach to the Ising
  Model: Foundations and Critical Point,'' {\em Journal of Statistical
  Physics}, vol.~152, no.~2, pp.~353--387, 2013.

\bibitem{BrydgesImbrie2003}
D.~C. Brydges and J.~Z. Imbrie, ``{Branched Polymers},'' {\em Annals of
  Mathematics}, vol.~158, pp.~1019--1039, 2003.

\bibitem{Viennot1986}
G.~Viennot, ``{Heaps of Pieces, I: Basic Definitions and Combinatorial
  Lemmas},'' in {\em Lecture Notes in Mathematics 1234: Combinatoire
  \'{e}num\'{e}rative} (A.~Dold and B.~Eckmann, eds.), pp.~321--350, Springer,
  1986.

\bibitem{CairnsMcIntyre1993}
M.~McIntyre and G.~Cairns, ``{A new formula for winding number},'' {\em
  Geometriae Dedicata}, vol.~46, pp.~149--159, May 1993.

\bibitem{ChelkakHonglerIzyurov2012}
D.~Chelkak, C.~Hongler, and K.~Izyurov, ``{Conformal Invariance of Spin
  Correlations in the Planar Ising Model},'' {\em arXiv:1202.2838v1}, 2012.

\bibitem{Reinhart1960}
B.~Reinhart, ``{The winding number on two manifolds},'' {\em Ann. Inst.
  Fourier}, vol.~10, pp.~271--283, 1960.

\bibitem{Reinhart1963}
B.~Reinhart, ``{Further remarks on the winding number},'' {\em Ann. Inst.
  Fourier}, vol.~1, pp.~155--160, 1963.

\bibitem{Chillingworth1972}
D.~R.~J. Chillingworth, ``{Winding numbers on surfaces, I},'' {\em
  Mathematische Annalen}, vol.~196, pp.~218--249, Sept. 1972.

\bibitem{Stembridge1990}
J.~Stembridge, ``{Nonintersecting Paths, Pfaffians, and Plane Partitions},''
  {\em Advances in Mathematics}, vol.~83, no.~1, pp.~96--131, 1990.

\bibitem{Fary1948}
I.~F\'{a}ry, ``{On Straight Line Representation of Planar Graphs},'' {\em Acta
  scientiarum mathematicarum (Szeged)}, vol.~11, p.~229, 1948.

\bibitem{Whitney1937}
H.~Whitney, ``{On Regular Closed Curves in the Plane},'' {\em Compositio Math},
  vol.~4, pp.~276--284, 1937.

\bibitem{Smale1958}
S.~Smale, ``{Regular Curves On Riemannian Manifolds},'' {\em Transactions of
  the American Mathematics Society}, vol.~87, no.~2, pp.~492--512, 1958.

\bibitem{FarkasKra1980}
H.~M. Farkas and I.~Kra, {\em {Riemann Surfaces}}.
\newblock Springer-Verlag New York, 1980.

\bibitem{Reinhart1959}
B.~Reinhart, ``{Line elements on the torus},'' {\em American Journal of
  Mathematics}, vol.~81, no.~3, pp.~617--631, 1959.

\bibitem{Krattenthaler2000}
C.~Krattenthaler, ``{The Theory of Heaps and the Cartier–-Foata Monoid},'' in
  {\em Commutation and Rearrangements} (P.~Cartier and D.~Foata, eds.),
  pp.~63--73, 2006.

\bibitem{Baxter1982}
R.J.~Baxter, {\em Exactly Solved Models in Statistical Physics}.
\newblock Academic New York, 1982.

\end{thebibliography}
\def\cprime{$'$}

\end{document}